\documentclass[letterpaper,12pt]{article}
\usepackage{amsmath,amssymb,amsthm,url}
\usepackage[letterpaper,hmargin=1in,vmargin=0.9in]{geometry}

\usepackage{tikz}
\usetikzlibrary{decorations.markings}
\tikzstyle{vertex}=[circle, draw, inner sep=0pt, minimum size=6pt]
\newcommand{\vertex}{\node[vertex]}

\usepackage{caption}

\usetikzlibrary{arrows,calc,patterns,positioning,shapes}
\usetikzlibrary{decorations.pathmorphing}
\tikzset{
coil/.style={decorate, decoration={coil,amplitude=4pt,segment length=5pt}},
snake/.style={decorate, decoration={snake}},
zigzag/.style={decorate, decoration={zigzag}}
}

\title{Block colourings of star systems}
\author{Robert F.\ Bailey and Iren Darijani\footnote{School of Science and the Environment (Mathematics), Memorial University--Grenfell Campus, 20 University Drive, Corner Brook NL A2H~6P9, Canada. E-mail: \texttt{rbailey@grenfell.mun.ca}, \texttt{i.darijani@mun.ca}} }

\newtheorem{thm}{Theorem}[section]

\newtheorem{cor}[thm]{Corollary}

\theoremstyle{definition}

\newtheorem{example}[thm]{Example}

\newcommand{\selda}{K\"{u}\c{c}\"{u}k\c{c}if\c{c}i}

\DeclareMathOperator{\Sym}{Sym}

\begin{document}

\maketitle

\begin{abstract}
An $e$-star system of order $n$ is a decomposition of the complete graph $K_n$ into copies of the complete bipartite graph $K_{1,e}$ (or $e$-star).  Such systems are known to exist if and only if $n\geq 2e$ and $e$ divides $\binom{n}{2}$.  We consider block colourings of such systems, where each $e$-star is assigned a colour, and two $e$-stars which share a vertex receive different colours.  We present a computer analysis of block colourings of small $3$-star systems. Furthermore, we prove that: (i) for $n\equiv 0,1$ mod $2e$ there exists either an $n$ or $(n-1)$-block colourable $e$-star system of order $n$; and (ii) when $e=3$, the same result holds in the remaining congruence classes mod $6$.
\end{abstract}

\section{Introduction} \label{sec:intro}

A {\em $G$-decomposition} of a graph $H$ is a pair $\mathcal{D}=(V,\mathcal{B})$, where $V$ is the vertex set of $H$, and $\mathcal{B}$ is a set of subgraphs of $H$, each isomorphic to $G$, whose edge sets partition the edge set of $H$.  The elements of $\mathcal{B}$ are known as the {\em blocks} of $\mathcal{D}$.  A {\em $G$-design} of order $n$ is a $G$-decomposition of the complete graph $K_n$ on $n$ vertices.  In the case where $G$ is a complete bipartite graph $K_{1,e}$, also known as an {\em $e$-star}, we call the design an {\em $e$-star system} of order $n$, denoted $S_e(n)$.

Necessary and sufficient conditions for the existence of $e$-star systems were determined by Yamamoto {\em et al.}~\cite{Yamamoto1975} in 1975: they showed that an $e$-star system of order $n$ exists if and only if (i) $n \geq 2e$, and (ii) $e$ divides $\binom{n}{2}$. We call a positive integer $n$ {\em admissible} if there exists an $e$-star system of order $n$.  Since a $1$-star is an edge and a $2$-star is a path, we will consider $e$-star systems for $e\geq 3$.

Our notation for an $e$-star isomorphic to $K_{1,e}$ will be $\{x; y_1,\ldots,y_e\}$, where $x$ is the vertex of degree $e$ (the {\em root} vertex), and $y_1,\ldots,y_e$ are the vertices of degree $1$ (the {\em pendant} vertices).

\begin{example} \label{3star6}
The following pair $(V,\mathcal{B})$ is a $3$-star system of order $6$, where $V=\{1,\ldots,6\}$ is the set of points and $\mathcal B=\big \{ \{1;3,5,6\}$, $\{2;1,3,6\}$, $\{4;1,2,3\}$, $\{5;2,3,4\}$, $\{6;3,4,5\}\big\}$ is the set of blocks ($3$-stars). 
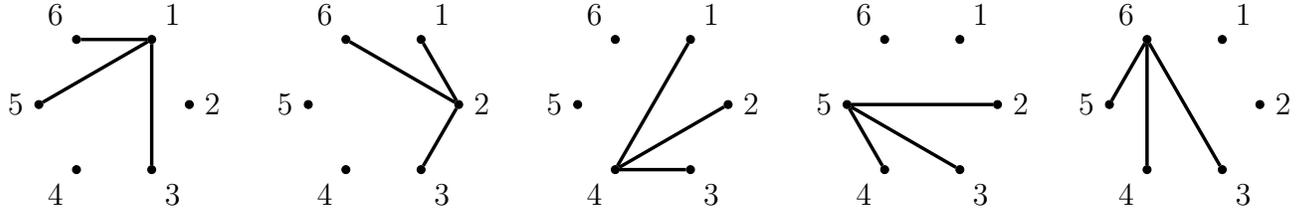
\vspace{0.5cm} 

 \begin{tikzpicture}[x=1cm, y=1cm]
 \tikzstyle{vertex}=[circle, draw, inner sep=0pt, minimum size=3pt]	
    \vertex (c1) at (60:1) [label=60:${1}$, fill=black]{};
 	\vertex (c6) at (120:1) [label=120:${6}$, fill=black]{};
        \vertex (c5) at (180:1) [label=180:${5}$, fill=black]{};
        \vertex (c4) at (240:1) [label=240:${4}$, fill=black]{};
        \vertex (c3) at (300:1) [label=300:${3}$, fill=black]{};
        \vertex (c2) at (360:1) [label=360:${2}$, fill=black]{};
		
 \path [line width=1.3pt]  
	(c1) edge (c3)
	(c1) edge (c5)
	(c1) edge (c6);	
\end{tikzpicture} \hspace{0.3cm}
\begin{tikzpicture}[x=1cm, y=1cm]
 \tikzstyle{vertex}=[circle, draw, inner sep=0pt, minimum size=3pt]	
    \vertex (c1) at (60:1) [label=60:${1}$, fill=black]{};
 	\vertex (c6) at (120:1) [label=120:${6}$, fill=black]{};
        \vertex (c5) at (180:1) [label=180:${5}$, fill=black]{};
        \vertex (c4) at (240:1) [label=240:${4}$, fill=black]{};
        \vertex (c3) at (300:1) [label=300:${3}$, fill=black]{};
        \vertex (c2) at (360:1) [label=360:${2}$, fill=black]{};

  \path [line width=1.3pt]   
	(c2) edge (c6)
	(c2) edge (c3)
         (c2) edge (c1);	
\end{tikzpicture}	\hspace{0.3cm}
\begin{tikzpicture}[x=1cm, y=1cm]
 \tikzstyle{vertex}=[circle, draw, inner sep=0pt, minimum size=3pt]	
    \vertex (c1) at (60:1) [label=60:${1}$, fill=black]{};
 	\vertex (c6) at (120:1) [label=120:${6}$, fill=black]{};
        \vertex (c5) at (180:1) [label=180:${5}$, fill=black]{};
        \vertex (c4) at (240:1) [label=240:${4}$, fill=black]{};
        \vertex (c3) at (300:1) [label=300:${3}$, fill=black]{};
        \vertex (c2) at (360:1) [label=360:${2}$, fill=black]{};
		
	 \path [line width=1.3pt] 	 
	(c4) edge (c3)
	(c4) edge (c2)
	(c4) edge (c1);		
\end{tikzpicture} \hspace{0.3cm}
 \begin{tikzpicture}[x=1cm, y=1cm]
 \tikzstyle{vertex}=[circle, draw, inner sep=0pt, minimum size=3pt]	
    \vertex (c1) at (60:1) [label=60:${1}$, fill=black]{};
 	\vertex (c6) at (120:1) [label=120:${6}$, fill=black]{};
        \vertex (c5) at (180:1) [label=180:${5}$, fill=black]{};
        \vertex (c4) at (240:1) [label=240:${4}$, fill=black]{};
        \vertex (c3) at (300:1) [label=300:${3}$, fill=black]{};
        \vertex (c2) at (360:1) [label=360:${2}$, fill=black]{};
	
 \path [line width=1.3pt]  
	(c5) edge (c4)
	(c5) edge (c3)
	(c5) edge (c2);	
\end{tikzpicture}
\hspace{0.3cm}\begin{tikzpicture}[x=1cm, y=1cm]
 \tikzstyle{vertex}=[circle, draw, inner sep=0pt, minimum size=3pt]	
    \vertex (c1) at (60:1) [label=60:${1}$, fill=black]{};
 	\vertex (c6) at (120:1) [label=120:${6}$, fill=black]{};
        \vertex (c5) at (180:1) [label=180:${5}$, fill=black]{};
        \vertex (c4) at (240:1) [label=240:${4}$, fill=black]{};
        \vertex (c3) at (300:1) [label=300:${3}$, fill=black]{};
        \vertex (c2) at (360:1) [label=360:${2}$, fill=black]{};
       
    \path [line width=1.3pt]     
       		(c6) edge (c5)
		(c6) edge (c4)
		(c6) edge (c3);
\end{tikzpicture} 

\end{example}

\subsection{Block-colourings}

A {\em block-colouring} of a $G$-design $\mathcal{D}=(V,\mathcal{B})$ is a partition of $\mathcal{B}$ into {\em colour classes}, where the blocks in each colour class are mutually disjoint.  (This amounts to colouring the edges of $K_n$ in such a way that the edges in a particular copy of $G$ all receive the same colour, and two edges incident with the same vertex are the same colour if and only if they lie in the same block.)  A design is said to be {\em $k$-block-colourable} if this is possible using $k$ colours; the $G$-design is {\em $k$-block-chromatic} if it is $k$-block-colourable but is not $(k-1)$-block-colourable. If a $G$-design $\mathcal{D}$ is $k$-block-chromatic, we say that its {\em chromatic index} (also known as its {\em block-chromatic number}) is $k$, and we denote this by $\chi'(\mathcal{D})$.  In other words, $\chi'(\mathcal{D})$ is the least integer $k$ for which $\mathcal{D}$ admits a $k$-block-colouring.

In this paper, we will consider properties of block-colourings of $e$-star systems for $e\geq 3$. To begin with, we note that as an $e$-star has $e+1$ vertices, we need that $n\geq 2(e+1)$ in order for two mutually-disjoint blocks to be able to exist; consequently, if $2e\leq n\leq 2e+1$ we will require each block to be assigned its own colour, and the block-chromatic number will trivially be equal to the number of blocks.  Thus the smallest non-trivial example is a block-colouring of a $3$-star system of order~$9$, as shown below.

\begin{example} \label{3star9-colour}
The following pair $(V,\mathcal{B})$ is an $S_3(9)$, where $V=\{1,\ldots,9\}$ and $\mathcal{B}=\big \{ \{1;3,5,6\}$, $\{2;1,3,6\}$, $\{4;1,2,3\}$, $\{5;2,3,4\}$, $\{6;3,4,5\}$, $\{7;1,2,3\}$, $\{8;4,5,9\}$, $\{7;4,5,8\}$, $\{8;1,2,3\}$, $\{6;7,8,9\}$, $\{9;1,2,3\}$, $\{9;4,5,7\}\big\}$. This system is $8$-block-colourable, with colour classes $\mathcal C^1=\big\{ \{1;3,5,6\}\big\}$, $\mathcal C^2=\big\{ \{2;1,3,6\},\{9;4,5,7\}\big\}$, $\mathcal C^3=\big\{\{4;1,2,3\},\{6;7,8,9\}\big\}$, $\mathcal C^4=\big\{ \{5;2,3,4\}\big\}$, $\mathcal C^5=\big \{\{6;3,4,5\}\big\}$, $\mathcal C^6=\big\{\{7;1,2,3\},\{8;4,5,9\}\big\}$, $\mathcal C^7=\big\{\{7;4,5,8\},\linebreak \{9,1,2,3\}\big\}$, and $\mathcal C^8=\big\{\{8;1,2,3\}\big\}$.
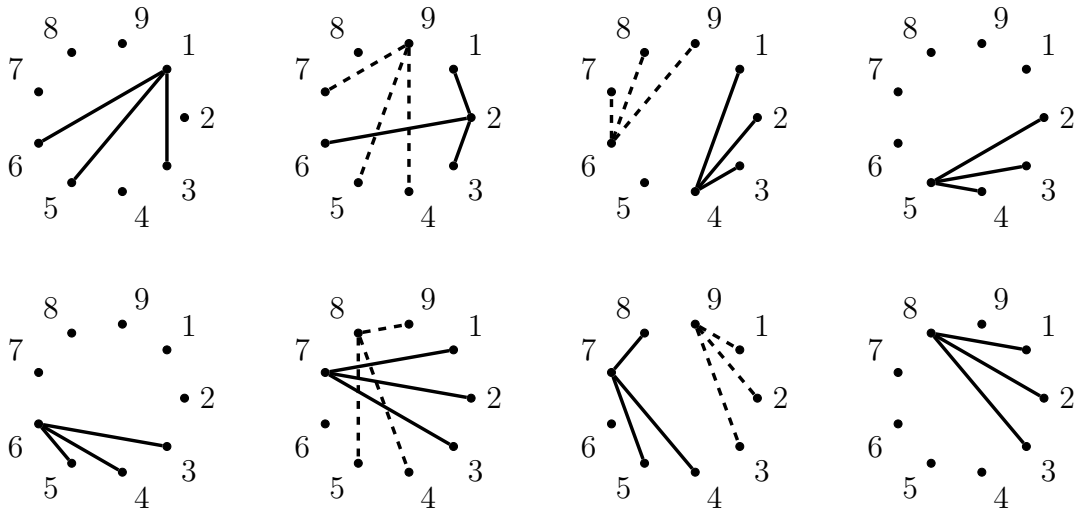
\vspace{0.5cm}

 \begin{tikzpicture}[x=1cm, y=1cm]
 \tikzstyle{vertex}=[circle, draw, inner sep=0pt, minimum size=3pt]	
    \vertex (c1) at (40:1) [label=40:${1}$, fill=black]{};
 	\vertex (c9) at (80:1) [label=80:${9}$, fill=black]{};
        \vertex (c8) at (120:1) [label=120:${8}$, fill=black]{};
        \vertex (c7) at (160:1) [label=160:${7}$, fill=black]{};
        \vertex (c6) at (200:1) [label=200:${6}$, fill=black]{};
        \vertex (c5) at (240:1) [label=240:${5}$, fill=black]{};
        \vertex (c4) at (280:1) [label=280:${4}$, fill=black]{};
        \vertex (c3) at (320:1) [label=320:${3}$, fill=black]{};
        \vertex (c2) at (360:1) [label=360:${2}$, fill=black]{};
		
 \path [line width=1.3pt]
                  (c1) edge (c3)
		(c1) edge (c5)
		(c1) edge (c6);
\end{tikzpicture} \hspace{0.5cm}
 \begin{tikzpicture}[x=1cm, y=1cm]
 \tikzstyle{vertex}=[circle, draw, inner sep=0pt, minimum size=3pt]	
    \vertex (c1) at (40:1) [label=40:${1}$, fill=black]{};
 	\vertex (c9) at (80:1) [label=80:${9}$, fill=black]{};
        \vertex (c8) at (120:1) [label=120:${8}$, fill=black]{};
        \vertex (c7) at (160:1) [label=160:${7}$, fill=black]{};
        \vertex (c6) at (200:1) [label=200:${6}$, fill=black]{};
        \vertex (c5) at (240:1) [label=240:${5}$, fill=black]{};
        \vertex (c4) at (280:1) [label=280:${4}$, fill=black]{};
        \vertex (c3) at (320:1) [label=320:${3}$, fill=black]{};
        \vertex (c2) at (360:1) [label=360:${2}$, fill=black]{};
		
 \path [line width=1.3pt]
                  (c2) edge (c1)
		(c2) edge (c3)
		(c2) edge (c6);
\path [line width=1.3pt]
                  (c9) edge[dashed] (c4)
		(c9) edge[dashed] (c5)
		(c9) edge[dashed] (c7);		
		
\end{tikzpicture} \hspace{0.5cm}
 \begin{tikzpicture}[x=1cm, y=1cm]
 \tikzstyle{vertex}=[circle, draw, inner sep=0pt, minimum size=3pt]	
    \vertex (c1) at (40:1) [label=40:${1}$, fill=black]{};
 	\vertex (c9) at (80:1) [label=80:${9}$, fill=black]{};
        \vertex (c8) at (120:1) [label=120:${8}$, fill=black]{};
        \vertex (c7) at (160:1) [label=160:${7}$, fill=black]{};
        \vertex (c6) at (200:1) [label=200:${6}$, fill=black]{};
        \vertex (c5) at (240:1) [label=240:${5}$, fill=black]{};
        \vertex (c4) at (280:1) [label=280:${4}$, fill=black]{};
        \vertex (c3) at (320:1) [label=320:${3}$, fill=black]{};
        \vertex (c2) at (360:1) [label=360:${2}$, fill=black]{};
		
 \path [line width=1.3pt]
                  (c4) edge (c1)
		(c4) edge (c2)
		(c4) edge (c3);
\path [line width=1.3pt]
                  (c6) edge[dashed] (c7)
		(c6) edge[dashed] (c8)
		(c6) edge[dashed] (c9);		
		
\end{tikzpicture} \hspace{0.5cm}
 \begin{tikzpicture}[x=1cm, y=1cm]
 \tikzstyle{vertex}=[circle, draw, inner sep=0pt, minimum size=3pt]	
    \vertex (c1) at (40:1) [label=40:${1}$, fill=black]{};
 	\vertex (c9) at (80:1) [label=80:${9}$, fill=black]{};
        \vertex (c8) at (120:1) [label=120:${8}$, fill=black]{};
        \vertex (c7) at (160:1) [label=160:${7}$, fill=black]{};
        \vertex (c6) at (200:1) [label=200:${6}$, fill=black]{};
        \vertex (c5) at (240:1) [label=240:${5}$, fill=black]{};
        \vertex (c4) at (280:1) [label=280:${4}$, fill=black]{};
        \vertex (c3) at (320:1) [label=320:${3}$, fill=black]{};
        \vertex (c2) at (360:1) [label=360:${2}$, fill=black]{};
		
 \path [line width=1.3pt]
                  (c5) edge (c2)
		(c5) edge (c3)
		(c5) edge (c4);	
		
\end{tikzpicture}
\vspace{0.5cm}

 \begin{tikzpicture}[x=1cm, y=1cm]
 \tikzstyle{vertex}=[circle, draw, inner sep=0pt, minimum size=3pt]	
    \vertex (c1) at (40:1) [label=40:${1}$, fill=black]{};
 	\vertex (c9) at (80:1) [label=80:${9}$, fill=black]{};
        \vertex (c8) at (120:1) [label=120:${8}$, fill=black]{};
        \vertex (c7) at (160:1) [label=160:${7}$, fill=black]{};
        \vertex (c6) at (200:1) [label=200:${6}$, fill=black]{};
        \vertex (c5) at (240:1) [label=240:${5}$, fill=black]{};
        \vertex (c4) at (280:1) [label=280:${4}$, fill=black]{};
        \vertex (c3) at (320:1) [label=320:${3}$, fill=black]{};
        \vertex (c2) at (360:1) [label=360:${2}$, fill=black]{};
		
 \path [line width=1.3pt]
                  (c6) edge (c4)
		(c6) edge (c3)
		(c6) edge (c5);	
		
\end{tikzpicture} \hspace{0.5cm}
 \begin{tikzpicture}[x=1cm, y=1cm]
 \tikzstyle{vertex}=[circle, draw, inner sep=0pt, minimum size=3pt]	
    \vertex (c1) at (40:1) [label=40:${1}$, fill=black]{};
 	\vertex (c9) at (80:1) [label=80:${9}$, fill=black]{};
        \vertex (c8) at (120:1) [label=120:${8}$, fill=black]{};
        \vertex (c7) at (160:1) [label=160:${7}$, fill=black]{};
        \vertex (c6) at (200:1) [label=200:${6}$, fill=black]{};
        \vertex (c5) at (240:1) [label=240:${5}$, fill=black]{};
        \vertex (c4) at (280:1) [label=280:${4}$, fill=black]{};
        \vertex (c3) at (320:1) [label=320:${3}$, fill=black]{};
        \vertex (c2) at (360:1) [label=360:${2}$, fill=black]{};
		
 \path [line width=1.3pt]
                  (c7) edge (c1)
		(c7) edge (c3)
		(c7) edge (c2);
\path [line width=1.3pt]
                  (c8) edge[dashed] (c4)
		(c8) edge[dashed] (c5)
		(c8) edge[dashed] (c9);		
		
\end{tikzpicture}  \hspace{0.5cm}
 \begin{tikzpicture}[x=1cm, y=1cm]
 \tikzstyle{vertex}=[circle, draw, inner sep=0pt, minimum size=3pt]	
    \vertex (c1) at (40:1) [label=40:${1}$, fill=black]{};
 	\vertex (c9) at (80:1) [label=80:${9}$, fill=black]{};
        \vertex (c8) at (120:1) [label=120:${8}$, fill=black]{};
        \vertex (c7) at (160:1) [label=160:${7}$, fill=black]{};
        \vertex (c6) at (200:1) [label=200:${6}$, fill=black]{};
        \vertex (c5) at (240:1) [label=240:${5}$, fill=black]{};
        \vertex (c4) at (280:1) [label=280:${4}$, fill=black]{};
        \vertex (c3) at (320:1) [label=320:${3}$, fill=black]{};
        \vertex (c2) at (360:1) [label=360:${2}$, fill=black]{};
		
 \path [line width=1.3pt]
                  (c7) edge (c4)
		(c7) edge (c5)
		(c7) edge (c8);
\path [line width=1.3pt]
                  (c9) edge[dashed] (c1)
		(c9) edge[dashed] (c2)
		(c9) edge[dashed] (c3);		
		
\end{tikzpicture} \hspace{0.5cm}
 \begin{tikzpicture}[x=1cm, y=1cm]
 \tikzstyle{vertex}=[circle, draw, inner sep=0pt, minimum size=3pt]	
    \vertex (c1) at (40:1) [label=40:${1}$, fill=black]{};
 	\vertex (c9) at (80:1) [label=80:${9}$, fill=black]{};
        \vertex (c8) at (120:1) [label=120:${8}$, fill=black]{};
        \vertex (c7) at (160:1) [label=160:${7}$, fill=black]{};
        \vertex (c6) at (200:1) [label=200:${6}$, fill=black]{};
        \vertex (c5) at (240:1) [label=240:${5}$, fill=black]{};
        \vertex (c4) at (280:1) [label=280:${4}$, fill=black]{};
        \vertex (c3) at (320:1) [label=320:${3}$, fill=black]{};
        \vertex (c2) at (360:1) [label=360:${2}$, fill=black]{};
		
 \path [line width=1.3pt]
                  (c8) edge (c1)
		(c8) edge (c2)
		(c8) edge (c3);
\end{tikzpicture}  
 
\end{example}

In subsection~\ref{sec:cliques}, we will see that this is the best possible: the least possible chromatic index for an $S_3(9)$ is $8$.

A $G$-design $\mathcal{D}=(V,\mathcal{B})$ is {\em resolvable} if it has a block-colouring where each colour class contains every vertex in $V$.  The existence of resolvable $G$-designs, also known as {\em $G$-factorizations}, is a major area of design theory with a long history dating back to the mid-19th century and the work of T.\ P.\ Kirkman.  Clearly, for a $G$-design of order $n$ to be resolvable, the number of vertices of $G$ must divide $n$.  In many well-known cases, this necessary condition is also sufficient: if $G=K_3$ (i.e.\ {\em Kirkman triple systems}), this was proved by Ray-Chaudhuri and Wilson in 1971~\cite{RC_wilson}; for $G=C_m$ (where $m\geq 3$ must be odd) by Alspach {\em et al.}\ in 1989 \cite{Alspach89}; for $G=P_k$ (i.e.\ a path with $k$ vertices), by Horton in 1985~\cite{Horton85} and Bermond, Heinrich and Yu in 1990~\cite{Bermond90}.  For the more general question of determining the least possible chromatic index of a $G$-design, see Vanstone {\em et al.} \cite{Vanstone93} for $G=K_3$, and Danziger, Mendelsohn and Quattrocchi \cite{Danziger04} for $G=P_3$ and $G=P_4$.  

In the case of resolvable $e$-star systems, the necessary conditions were obtained by Huang in 1976~\cite{Huang76}: for such as system of order $n$ to exist, we must have that $n\equiv 0$ (mod~$e+1$) and $n\equiv 1$ (mod~$2e$); clearly these cannot be satisfied when $e$ is odd.  When $e$ is even, these conditions were shown to be sufficient by Yu in 1993~\cite{Yu93}.  More recently, an elementary proof of the non-existence of resolvable $3$-star systems was given by \selda\ {\em et al.}\ in 2015~\cite{resolvable}.  However, as far as the authors are aware, the more general question of determining the least possible chromatic index of an $e$-star system remains open.  This paper is devoted to investigating this.

A trivial lower bound on the chromatic index of a $G$-design can be obtained by dividing the number of blocks by the largest possible size of a colour class (i.e.\ the maximum possible number of disjoint blocks).  For an $e$-star system, the maximum number of disjoint blocks is $\lfloor n/(e+1) \rfloor$, and the number of blocks is $n(n-1)/2e$.  Thus, for an $e$-star system $\mathcal{D}$ of order $n$, we have
\[ \chi'(\mathcal{D}) \geq \left\lceil \frac{n(n-1)}{2e} \middle/ \left\lfloor \frac{n}{e+1} \right\rfloor \right\rceil. \]
We will denote this lower bound by $L(n,e)$.  When the necessary conditions for the existence of a resolvable $e$-star system are satisfied, the floor and ceiling functions disappear, and we are left with the obvious formula for the number of parallel classes.

\subsection{Other notions of colouring}
We remark that there are a number of notions of colourings of designs, which generalize vertex- and edge-colourings of graphs; block-colourings, as considered in this paper, are a natural analogue of edge-colourings.  As a generalization of vertex-colouring, a $G$-design $(V,\mathcal{B})$ is said to be {\em weakly $k$-colourable} if $V$ can be partitioned into $k$ colour classes such that no subgraph in $\mathcal{B}$ is monochromatic; a $G$-design is {\em $k$-chromatic} if it is $k$-colourable but not $(k-1)$-colourable.  Weak colourings of $G$-designs have been studied for many classes of $G$-designs; in the case of $e$-star systems, these were considered in a 2020 paper of Pike and the first author~\cite{DarijaniPike}.

\section{Experimental results} \label{sec:computer}

To gain an understanding of any class of combinatorial objects, it is often desirable to consider small cases computationally.  The chromatic index of star systems is no exception to this; as such, we performed a number of computer experiments using the {\sf GAP} computer algebra system~\cite{GAP} to study small $3$-star systems, and determine their chromatic index.  We used two different approaches: the first obtained $3$-star systems as cliques in a suitably constructed graph; the second was to view $3$-star systems as a particular class of block designs, and use classification tools for those.  In both cases, we can then obtain the chromatic index of each $3$-star system obtained, by first forming its block intersection graph, and then finding the chromatic number of this graph.

\subsection{Clique-finding} \label{sec:cliques}

A commonly-used technique to construct combinatorial objects is to devise a graph, and then to search this graph for a clique which corresponds to the desired object (see the survey by \"Osterg{\aa}rd \cite{Ostergard05} for some examples of this).  We were able to use this approach successfully to obtain a complete classification of $3$-star systems of order~$9$.

We construct a graph $\Gamma$ as follows: the vertex set of $\Gamma$ will be the set of all possible $3$-stars on a set $9$ vertices, of which there are $9\cdot\binom{8}{3}=504$ ($9$ choices for the root vertex, and $\binom{8}{3}$ choices for the pendant vertices), and vertices in $\Gamma$ will be adjacent if and only if the corresponding $3$-stars are edge-disjoint.  This graph naturally admits a vertex-transitive action of the symmetric group $\Sym(9)$ (which, in fact, turns out to be the full automorphism group of $\Gamma$), as the vertex set consists of the images of elements of $\Sym(9)$ applied to the `canonical' $3$-star $\{1; 2,3,4\}$.  A clique in $\Gamma$ corresponds to a set of mutually edge-disjoint $3$-stars in $K_9$.  Furthermore, if there exists a clique of size $\frac{1}{3}\binom{9}{2}=12$, then this corresponds to a $3$-star system of order $9$ (so this is the maximum possible size of a clique in $\Gamma$).

Using the {\sf GRAPE} package~\cite{GRAPE} for the {\sf GAP} system, it is straightforward to construct the graph $\Gamma$.  Also, {\sf GRAPE} has an in-built command, {\footnotesize \texttt{CliquesOfGivenSize}}, to find cliques of a specified size in a given graph; in particular, it can enumerate representatives of all orbits on cliques for some prescribed group acting on the vertices.  We applied this command to the graph $\Gamma$, and found a total of $51,770$ $12$-cliques in $\Gamma$ arising from the action of $\Sym(9)$.  (This computation took approximately 75 minutes on an 3.20GHz Intel i7 CPU with 16GB RAM.)

The {\sf GRAPE} package also now includes an efficient function to determine the chromatic number of a given graph.  From a given $3$-star system $\mathcal{D}$, the {\em block-intersection graph} $\mathrm{BIG}(\mathcal{D})$ of $\mathcal{D}$ has the blocks (i.e.\ $3$-stars) of $\mathcal{D}$ as its vertices, and two blocks are adjacent if and only if their intersection is non-empty; a block-colouring of $\mathcal{D}$ is clearly equivalent to a proper vertex-colouring of $\mathrm{BIG}(\mathcal{D})$.  It is also straightforward to construct $\mathrm{BIG}(\mathcal{D})$ in {\sf GRAPE}; we did this for each of the $51,770$ $S_3(9)$ obtained above, and determined the chromatic index for each of them.  The results of this computation (which took around a further 30 minutes) are in Table~\ref{table:3star9} below.  The code used is given in Appendices A.1 and A.2.

\begin{table}[hbtp]
\centering
\renewcommand{\arraystretch}{1.25}
\begin{tabular}{ccc}
Chromatic index & \quad & \# systems \\ \hline
$8$   && $\phantom{9}2,192$ \\
$9$   && $12,221$ \\
$10$  && $21,420$ \\
$11$  && $13,352$ \\
$12$  && $\phantom{9}2,585$ \\ \hline
Total && $51,770$
\end{tabular}
\renewcommand{\arraystretch}{1.0}
\caption{Chromatic index for $3$-star systems of order $9$}
\label{table:3star9}
\end{table}

In particular, we determined that the least possible chromatic index for an $S_3(9)$ is $8$, as with the system in Example~\ref{3star9-colour}; the trivial lower bound is $L(9,3)=6$, so this bound is not achieved here.  Table~\ref{table:3star9} also shows that there exist examples of $S_3(9)$ where each block requires its own colour (as the chromatic index, $12$, is equal to the number of blocks); in other words, any two $3$-stars in such a system must have a vertex in common, as in Example~\ref{rainbow} below.

\begin{example} \label{rainbow}
The following are the blocks of a $12$-block-chromatic $S_3(9)$ with vertex set $\{1,\ldots,12\}$:
\renewcommand{\arraystretch}{1.25}
\[ \begin{array}{ccccccc}
\{ 1; 2, 3, 4 \} &\quad&
\{ 1; 5, 6, 7 \} &\quad&
\{ 2; 3, 4, 5 \} &\quad&
\{ 3; 4, 5, 6 \} \\
\{ 3; 7, 8, 9 \} &&
\{ 4; 5, 6, 7 \} &&
\{ 6; 2, 5, 7 \} &&
\{ 7; 2, 5, 8 \} \\
\{ 8; 1, 4, 5 \} &&
\{ 8; 2, 6, 9 \} &&
\{ 9; 1, 5, 6 \} &&
\{ 9; 2, 4, 7 \}
\end{array} \]
\renewcommand{\arraystretch}{1.0}
\end{example}

We remark that we did not test the systems we obtained for isomorphism, so the number of isomorphism classes of $3$-star systems with given chromatic index may be smaller than the number of orbit representatives listed in Table~\ref{table:3star9}.

\subsection{Classifying designs} \label{sec:soicher}
Unfortunately, the clique-finding approach of Section~\ref{sec:cliques} was not feasible in the next case, namely $3$-star systems of order $10$, as the size of the search space is too large, so a different approach was required.  The authors are indebted to Leonard Soicher for suggesting this method, which makes use of the {\sf DESIGN} package for {\sf GAP} \cite{DESIGN}.

A $3$-star system $\mathcal{D}$ of order $n$ can be interpreted as a block design $\mathcal{D}_E$, where the points of $\mathcal{D}_E$ are the edges of $K_n$, and each block of $\mathcal{D}_E$ consists of the three edges of a $3$-star in $\mathcal{D}$.  (In fact, any $G$-design may be interpreted in such a way.)  The value of this interpretation is that the {\sf DESIGN} package includes a function, {\footnotesize \texttt{BlockDesigns}}, for classifying designs up to isomorphism, which can be applied here.  This classification function is most effective when searching for designs invariant under some prescribed group of automorphisms; we applied it to classify $3$-star systems of admissible order $n$ for $10\leq n\leq 16$ invariant under certain cyclic groups of prime order $p\leq n$.  Once the systems were obtained, we could then find their chromatic index by using {\sf GRAPE} in the same way as in the previous section.  The code used is given in Appendices A.1 and A.3.

Our results are given in two tables.  First, in Table~\ref{table:invariant-systems} we give the number of systems $S_3(n)$ invariant under each group.  Note that, for a given prime $p$, the symmetric group $\Sym(n)$ has exactly $\lfloor n/p \rfloor$ conjugacy classes of subgroups of order $p$, corresponding to the possible cycle types of its non-identity elements (i.e.\ the number of disjoint $p$-cycles); each such subgroup may be distinguished by the order of its normalizer $N$ in $\Sym(n)$.  Second, in Table~\ref{table:chromatic} we will give the number of systems with each chromatic index for each order.

\begin{table}[hbtp]
\centering
\renewcommand{\arraystretch}{1.25}
\begin{tabular}{cccc}
$n$ & $p$ & $|N|$   & \# systems \\ \hline
10  &   2 & various & 0 \\
		&   3 & 30240   & 583 \\
		&   3 & 864     & 0 \\
    &   3 & 324     & 800 \\
		&   5 & 2400    & 0 \\ 
		&   5 & 200     & 40 \\ 
		&   7 & 252     & 0 \\ \hline
12  &   2 & various & 0 \\
    &   3 & 25920  & 49816 \\
		&   3 & 3888   & 0 \\
		&   3 & 1944   & 0 \\ 
    &   5 & 100800 & 708 \\
    &   5 & 400    & 0 \\
		&   7 & 5040   & 0 \\
		&  11 & 110    & 32 \\ \hline
13  &   5 & 1200   & 0 \\
    &  11 & 220    & 0 \\
		&  13 & 156    & 54 \\ \hline
15  &   3 & 23328  & 0 \\
    &   5 & 24000  & 0 \\
		&   7 & 588    & 89600 \\
		&  13 & 312    & 0 \\ \hline
16  
    &   7 & 1176   & 0 \\
		&  11 & 13200  & 0 \\
		&  13 & 936    & 0 \\ \hline
\end{tabular}
\renewcommand{\arraystretch}{1.0}
\caption{Numbers of $S_3(n)$ invariant under given cyclic groups of prime order, for $10\leq n\leq 16$}
\label{table:invariant-systems}
\end{table}

\begin{table}[hbtp]
\centering
\renewcommand{\arraystretch}{1.25}
\begin{tabular}{ccccc}
$n$ & \quad & Chromatic index & \quad & \# systems \\ \hline
10 &&  8 && 19 \\
   &&  9 && 389 \\
	 && 10 && 501 \\
	 && 11 && 41 \\
	 && 12 && 429 \\
	 && 13 && 15 \\
	 && 14 && 9 \\
	 && 15 && 20 \\ \hline
12 && 10 && 823 \\
   && 11 && 19953 \\
   && 12 && 22828 \\ 
	 && 13 && 5880 \\
	 && 14 && 1036 \\
	 && 15 && 36 \\ \hline
13 && 10 && 5 \\
   && 11 && 8 \\
	 && 12 && 0 \\
	 && 13 && 39 \\
	 && 14 && 2 \\ \hline
15 && 14 && 961 \\
   && 15 && 17 \\
	 && 16 && 14 \\
	 && 17 && 6 \\
	 && 18 && 2 \\ \hline
\end{tabular}
\renewcommand{\arraystretch}{1.0}
\caption{Chromatic index for $3$-star systems of order $n$, for $10\leq n\leq 15$.  For $n=15$, a random sample of $1000$ systems was chosen.}
\label{table:chromatic}
\end{table}

Some interesting observations can be made from these tables.  First, we note that the trivial lower bound is actually attained when $n=10$; we have $L(10,3)=8$, and we found $19$ examples of $S_3(10)$ with chromatic index $8$.  An example of such a system, along with an $8$-block colouring, is given in Example~\ref{optimal}.  However, we do not observe this for the other orders we considered (note that $L(12,3)=8$, $L(13,3)=9$, $L(15,3)=12$ and $L(16,3)=10$).  Also when $n=10$, the trivial upper bound is also achieved (as happened with $n=9$), so there exist $S_3(10)$ where any two blocks intersect; however, we did not observe this for larger orders (an $S_3(12)$ has $22$ blocks, and an $S_3(13)$ has $26$).  

\begin{example} \label{optimal}
The following are the colour classes of an $8$-block chromatic $S_3(10)$:
\[ \begin{array}{lll}
\big\{ \{ 1; 2,4,5 \},\, \{6; 7,8,10\} \big\} & \quad & \big\{ \{2; 3,4,5\},\, \{9; 6,8,10\} \big\} \\
\big\{ \{ 4; 5,6,7 \},\, \{10; 1,2,3\} \big\} & \quad & \big\{ \{4; 8,9,10\},\, \{6; 1,2,3\} \big\} \\
\big\{ \{ 5; 6,9,10 \},\, \{7; 1,2,3\} \big\} & \quad & \big\{ \{7; 5,9,10\},\, \{8; 1,2,3\} \big\} \\
\big\{ \{ 8; 5,7,10 \},\, \{9; 1,2,3\} \big\} & \quad & \big\{ \{3; 1,4,5\} \big\}
\end{array} \]
\end{example}

For each value of $n$, though, we find plenty of examples of chromatic index $n$ or less.  This suggests that $n$ may be an upper bound on the least possible chromatic index for an $S_3(n)$, or possibly for an $S_e(n)$ more generally (i.e.\ that there should exist $n$-block-colourable $S_e(n)$ for a fixed value of $e$ and any admissible value of $n$).


\section{An upper bound: the case of $e$-star systems} \label{sec:bound-e}
In this section, we investigate block-colourings of $e$-star systems for arbitrary $e \geq 3$. We show that, for all $n \equiv 0,1$ (mod $2e$), there exists an $e$-star system which admits a block colouring using either $n$ or $n-1$ colours; this gives an upper bound on the least possible chromatic index for $e$-star systems of order $n$ where $n \equiv 0,1$ (mod $2e$). In the theorems below, we give these constructions, which depend on congruence classes modulo $4e$.  We do not claim that the systems we construct have chromatic index $n$ or $n-1$, merely that they admit a colouring with that number of colours.  

We begin with the most fundamental case, namely when $n\equiv 0$ (mod $4e$); the other cases all involve extensions or adaptations of the construction given here.

\begin{thm}\label{cong0_2e_even}
For $n\equiv 0$ (mod $4e$), there exists an $(n-1)$-block-colourable $e$-star system of order $n$.
\end{thm}

\begin{proof}
Let $n=4et$, where $t\geq 1$. Let $V=\{v_1^1,\ldots,v_{2e}^1,v_1^2,\ldots,v_{2e}^2,\ldots,v_1^{2t},\ldots,v_{2e}^{2t}\}$ be the set of points. Partition $V$ into $2t$ subsets $V_1=\{v_1^1,\ldots,v_{2e}^1\}$, $V_2=\{v_1^2,\ldots,v_{2e}^2\}$,$\ldots$, $V_{2t}=\{v_1^{2t},\ldots,v_{2e}^{2t}\}$ of size $2e$.  On each $V_i$ (for $1\leq i\leq 2t$), we place a copy of an $e$-star system $(V_i,\mathcal B_i)$ of order $2e$, where $\mathcal B_i=\{B_i^1,\ldots,B_i^{2e-1}\}$; this is necessarily $(2e-1)$-block-chromatic. 

Next, we construct a complete graph $K_{2t}$ in which the vertices are $V_1,\ldots,V_{2t}$. Since $2t$ is even, $K_{2t}$ admits a 1-factorization, and hence we can partition the set of all pairs of subsets $V_1,\ldots,V_{2t}$ into $2t-1$ $1$-factors $F_1,\ldots,F_{2t-1}$. Without loss of generality, we may assume that $F_1=\big\{(V_1,V_2),(V_3,V_4),\ldots,(V_{2t-1},V_{2t})\big\}$.  Furthermore, we will assume that each pair is ordered.  Then for each $1\leq j\leq t$, we form a collection of pairs of $e$-stars as follows:%
\renewcommand{\arraystretch}{1.25}%
\[ \begin{array}{l}
\mathcal F_{2j-1}^1=\big\{ \{v_1^{2j-1};v_1^{2j},\ldots,v_e^{2j}\}, \{v_2^{2j-1};v_{e+1}^{2j},\ldots,v_{2e}^{2j}\}\big\},\\ 
\mathcal F_{2j-1}^2=\big\{\{v_1^{2j-1};v_{e+1}^{2j},\ldots,v_{2e}^{2j}\}, \{v_2^{2j-1};v_1^{2j},\ldots,v_e^{2j}\}\big\},\\
\phantom{\mathcal F_{2j-1}^2}\, \vdots \\
\mathcal F_{2j-1}^{2e-1}=\big\{ \{v_{2e-1}^{2j-1};v_{1}^{2j},\ldots,v_e^{2j}\}, \{v_{2e}^{2j-1};v_{e+1}^{2j},\ldots,v_{2e}^{2j}\}\big\},\\
\mathcal F_{2j-1}^{2e}=\big\{\{v_{2e-1}^{2j-1};v_{e+1}^{2j},\ldots,v_{2e}^{2j}\}, \{v_{2e}^{2j-1};v_1^{2j},\ldots,v_e^{2j}\}\big\}.
\end{array} \]%
\renewcommand{\arraystretch}{1.25}%
Then the edges between pairs of subsets in the $1$-factor $F_1$ can be partitioned into $2e$ colour classes 
\[ \mathcal C_1^1=\bigcup\limits_{j=1}^{t} \mathcal F_{2j-1}^1,\quad
   \mathcal C_1^2=\bigcup\limits_{j=1}^{t} \mathcal F_{2j-1}^2,\quad
   \ldots, \quad 
   \mathcal C_1^{2e}=\bigcup\limits_{j=1}^{t} \mathcal F_{2j-1}^{2e}. 
\]
For each $i$ where $2\leq i\leq 2t-1$, the edges between pairs of subsets in the $1$-factor $F_i$ can be partitioned into $2e$ colour classes $\mathcal C_i^2,\ldots, \mathcal C_i^{2e}$ in a similar manner.  This uses a total of $2e(2t-1)$ distinct colours.

It remains to assign colours to the $e$-stars in each $\mathcal{B}_i$.  Since these form a collection of $e$-star systems of order $2e$ which are mutually disjoint, we may use the same set of $2e-1$ colours each time.  So, for each $k$ where $1\leq k\leq 2e-1$, we define colour classes $\mathcal D^k=\{B_1^k,\ldots,B_{2t}^k\}$.

Pulling all of this together, we have an $e$-star system $(V,\mathcal B)$ of order $4et$, where $\mathcal B=\big(\bigcup\limits_{i=1}^{2t} \mathcal B_i\big) \cup \big(\bigcup_{i=1}^{2t-1} \mathcal C_i^1\big) \cup \cdots \big(\bigcup_{i=1}^{2t-1} \mathcal C_i^{2e}\big)$, with $2e(2t-1)+2e-1 = 4et-1$ colour classes, $\mathcal D^1, \ldots, \mathcal D^{2e-1}, \mathcal C_1^1,\ldots, \mathcal C_1^{2e},\ldots, \mathcal C_{2t-1}^1,\ldots, \mathcal C_{2t-1}^{2e}.$  This completes the proof.
\end{proof}

To illustrate this construction, we give the example below.

\begin{example}
We will construct a $23$-block-colourable $3$-star system of order $24$ using the method of Theorem~\ref{cong0_2e_even}. Let $V=\{1,\ldots,24\}$ be the set of points. We partition $V$ into four subsets $V_1=\{1,\ldots,6\}$, $V_2=\{7,\ldots,12\}$, $V_3=\{13,\ldots,18\}$, and $V_4=\{19,\ldots,24\}$. For each $i$ where $1\leq i\leq 4$, we place a copy of the $3$-star system of order $6$ given in Example~\ref{3star6}, which is $5$-block-chromatic; label the blocks of these as $\mathcal B_i=\{B_i^1,B_i^2,B_i^3,B_i^4,B_i^5\}$.

Next, we construct a complete graph $K_4$ with vertices $V_1,V_2,V_3,V_4$.
\begin{center}
 \begin{tikzpicture}[x=1.2cm, y=1.2cm]
 \tikzstyle{vertex}=[circle, draw, inner sep=0pt, minimum size=4pt]	
    \vertex (c1) at (90:1) [label=90:${V_1}$, fill=black]{};
 	\vertex (c4) at (180:1) [label=180:${V_4}$, fill=black]{};
        \vertex (c3) at (270:1) [label=270:${V_3}$, fill=black]{};
        \vertex (c2) at (360:1) [label=360:${V_2}$, fill=black]{};
        	
 \path [line width=1pt]  
	(c1) edge (c2)
	(c1) edge (c3)
	(c1) edge (c4)
	(c2) edge (c3)
	(c2) edge (c4)
	(c3) edge (c4);		
\end{tikzpicture} 
\end{center}
This $K_4$ admits a $1$-factorization with $1$-factors $F_1,F_2,F_3$, where $F_1=\{(V_1,V_2),(V_3,V_4)\}$, $F_2=\{(V_1,V_3),(V_2,V_4)\}$, and $F_3=\{(V_1,V_4),(V_2,V_3)\}$. The edges between $V_1$ and $V_2$ in $F_1$ can be partitioned into six colour classes $\mathcal F_1^1, \ldots, \mathcal F_1^6$ as depicted below.
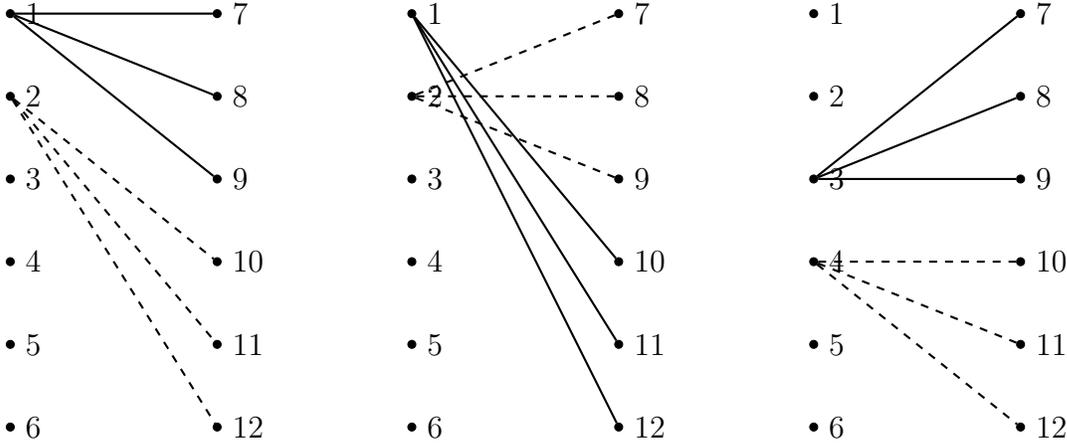
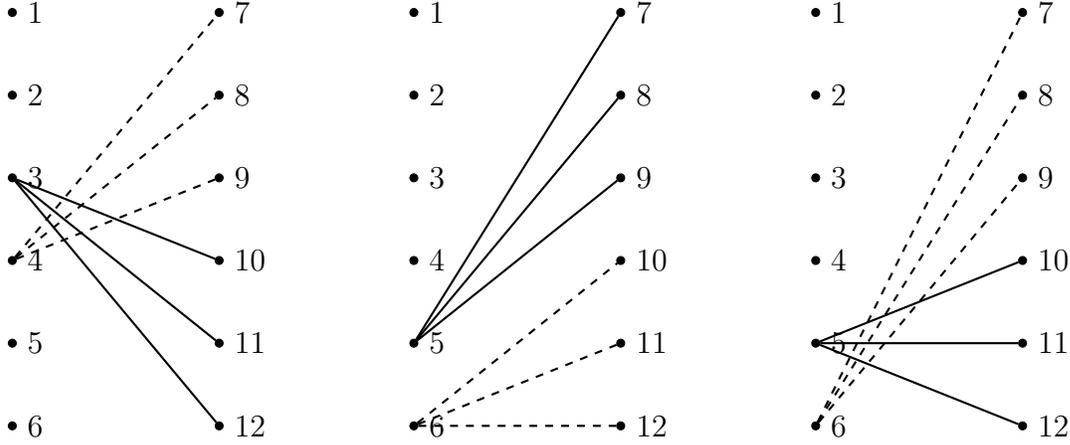
\vspace{0.5cm}

\begin{tikzpicture}[scale=0.55]
  
   \tikzstyle{vertex}=[circle, draw, inner sep=0pt, minimum size=3pt, fill=black]
      
   \vertex (n6) at (1,2)[label=0:$6$] {};
  \vertex (n5) at (1,4) [label=0:$5$] {};
  \vertex (n4) at (1,6) [label=0:$4$] {};
  \vertex (n3) at (1,8)[label=0:$3$] {};
 \vertex (n2) at (1,10)[label=0:$2$]  {};
  \vertex (n1) at (1,12)[label=0:$1$]{};

   \vertex (n12) at (6,2)[label=0:$12$] {};
  \vertex (n11) at (6,4) [label=0:$11$] {};
  \vertex (n10) at (6,6) [label=0:$10$] {};
  \vertex (n9) at (6,8)[label=0:$9$] {};
 \vertex (n8) at (6,10)[label=0:$8$]  {};
  \vertex (n7) at (6,12)[label=0:$7$]{};

  \path[thick] 
  (n1) edge (n7)
  (n1) edge (n8)
  (n1) edge (n9)
  (n2) edge[dashed] (n10)
  (n2) edge[dashed] (n11)
  (n2) edge[dashed] (n12);
 
  \end{tikzpicture} \hspace{1.5cm}
\begin{tikzpicture}[scale=0.55]
  
   \tikzstyle{vertex}=[circle, draw, inner sep=0pt, minimum size=3pt, fill=black]
      
   \vertex (n6) at (1,2)[label=0:$6$] {};
  \vertex (n5) at (1,4) [label=0:$5$] {};
  \vertex (n4) at (1,6) [label=0:$4$] {};
  \vertex (n3) at (1,8)[label=0:$3$] {};
 \vertex (n2) at (1,10)[label=0:$2$]  {};
  \vertex (n1) at (1,12)[label=0:$1$]{};
      
   \vertex (n12) at (6,2)[label=0:$12$] {};
  \vertex (n11) at (6,4) [label=0:$11$] {};
  \vertex (n10) at (6,6) [label=0:$10$] {};
  \vertex (n9) at (6,8)[label=0:$9$] {};
 \vertex (n8) at (6,10)[label=0:$8$]  {};
  \vertex (n7) at (6,12)[label=0:$7$]{};

  \path[thick] 
  (n1) edge (n10)
  (n1) edge (n11)
  (n1) edge (n12)
  (n2) edge[dashed] (n7)
  (n2) edge[dashed] (n8)
  (n2) edge[dashed] (n9);
 
  \end{tikzpicture} \hspace{1.5cm}
\begin{tikzpicture}[scale=0.55]
  
   \tikzstyle{vertex}=[circle, draw, inner sep=0pt, minimum size=3pt, fill=black]
      
   \vertex (n6) at (1,2)[label=0:$6$] {};
  \vertex (n5) at (1,4) [label=0:$5$] {};
  \vertex (n4) at (1,6) [label=0:$4$] {};
  \vertex (n3) at (1,8)[label=0:$3$] {};
 \vertex (n2) at (1,10)[label=0:$2$]  {};
  \vertex (n1) at (1,12)[label=0:$1$]{};
      
   \vertex (n12) at (6,2)[label=0:$12$] {};
  \vertex (n11) at (6,4) [label=0:$11$] {};
  \vertex (n10) at (6,6) [label=0:$10$] {};
  \vertex (n9) at (6,8)[label=0:$9$] {};
 \vertex (n8) at (6,10)[label=0:$8$]  {};
  \vertex (n7) at (6,12)[label=0:$7$]{};

  \path[thick] 
  (n3) edge (n7)
  (n3) edge (n8)
  (n3) edge (n9)
  (n4) edge[dashed] (n10)
  (n4) edge [dashed](n11)
  (n4) edge[dashed] (n12);
 
  \end{tikzpicture}
 \vspace{1.5cm} 
  
 \begin{tikzpicture}[scale=0.55]
  
   \tikzstyle{vertex}=[circle, draw, inner sep=0pt, minimum size=3pt, fill=black]
      
   \vertex (n6) at (1,2)[label=0:$6$] {};
  \vertex (n5) at (1,4) [label=0:$5$] {};
  \vertex (n4) at (1,6) [label=0:$4$] {};
  \vertex (n3) at (1,8)[label=0:$3$] {};
 \vertex (n2) at (1,10)[label=0:$2$]  {};
  \vertex (n1) at (1,12)[label=0:$1$]{};
      
   \vertex (n12) at (6,2)[label=0:$12$] {};
  \vertex (n11) at (6,4) [label=0:$11$] {};
  \vertex (n10) at (6,6) [label=0:$10$] {};
  \vertex (n9) at (6,8)[label=0:$9$] {};
 \vertex (n8) at (6,10)[label=0:$8$]  {};
  \vertex (n7) at (6,12)[label=0:$7$]{};

  \path[thick] 
  (n3) edge (n10)
  (n3) edge (n11)
  (n3) edge (n12)
  (n4) edge[dashed](n7)
  (n4) edge[dashed] (n8)
  (n4) edge[dashed] (n9);
 
  \end{tikzpicture} \hspace{1.5cm}  
\begin{tikzpicture}[scale=0.55]
  
   \tikzstyle{vertex}=[circle, draw, inner sep=0pt, minimum size=3pt, fill=black]
      
   \vertex (n6) at (1,2)[label=0:$6$] {};
  \vertex (n5) at (1,4) [label=0:$5$] {};
  \vertex (n4) at (1,6) [label=0:$4$] {};
  \vertex (n3) at (1,8)[label=0:$3$] {};
 \vertex (n2) at (1,10)[label=0:$2$]  {};
  \vertex (n1) at (1,12)[label=0:$1$]{};
      
   \vertex (n12) at (6,2)[label=0:$12$] {};
  \vertex (n11) at (6,4) [label=0:$11$] {};
  \vertex (n10) at (6,6) [label=0:$10$] {};
  \vertex (n9) at (6,8)[label=0:$9$] {};
 \vertex (n8) at (6,10)[label=0:$8$]  {};
  \vertex (n7) at (6,12)[label=0:$7$]{};

  \path[thick] 
  (n6) edge[dashed] (n10)
  (n6) edge[dashed] (n11)
  (n6) edge[dashed] (n12)
  (n5) edge (n7)
  (n5) edge (n8)
  (n5) edge (n9);
 
  \end{tikzpicture} \hspace{1.5cm}
  \begin{tikzpicture}[scale=0.55]
  
   \tikzstyle{vertex}=[circle, draw, inner sep=0pt, minimum size=3pt, fill=black]
      
   \vertex (n6) at (1,2)[label=0:$6$] {};
  \vertex (n5) at (1,4) [label=0:$5$] {};
  \vertex (n4) at (1,6) [label=0:$4$] {};
  \vertex (n3) at (1,8)[label=0:$3$] {};
 \vertex (n2) at (1,10)[label=0:$2$]  {};
  \vertex (n1) at (1,12)[label=0:$1$]{};
      
   \vertex (n12) at (6,2)[label=0:$12$] {};
  \vertex (n11) at (6,4) [label=0:$11$] {};
  \vertex (n10) at (6,6) [label=0:$10$] {};
  \vertex (n9) at (6,8)[label=0:$9$] {};
 \vertex (n8) at (6,10)[label=0:$8$]  {};
  \vertex (n7) at (6,12)[label=0:$7$]{};

  \path[thick] 
  (n5) edge (n10)
  (n5) edge (n11)
  (n5) edge (n12)
  (n6) edge[dashed] (n7)
  (n6) edge[dashed] (n8)
  (n6) edge [dashed](n9);
  \end{tikzpicture} 
  \vspace{0.5cm}
  
Likewise, we can partition the edges between $V_3$ and $V_4$ into six colour classes $\mathcal F_3^1,\ldots, \mathcal F_3^6$ in a similar manner. Then the edges between pairs of subsets in the $1$-factor $F_1$ are partitioned into six colour classes $\mathcal C_1^1=\mathcal F_1^1 \cup \mathcal F_3^1,\ldots,\mathcal C_1^6=\mathcal F_1^6 \cup \mathcal F_3^6$.  Using the same approach, we obtain another six colour classes, $\mathcal C_2^1, \ldots,\mathcal C_2^6$, from the $1$-factor $F_2$, and a further six colour classes, $\mathcal C_3^1, \ldots,\mathcal C_3^6$, from the $1$-factor $F_3$.

This leaves the four copies of the $3$-star system of order $6$.  For each $k$ where $1\leq k\leq 5$, we let $\mathcal D^k =\{B_1^k,B_2^k,B_3^k,B_4^k\}$.  For instance, $\mathcal D_1$ is shown below:\\

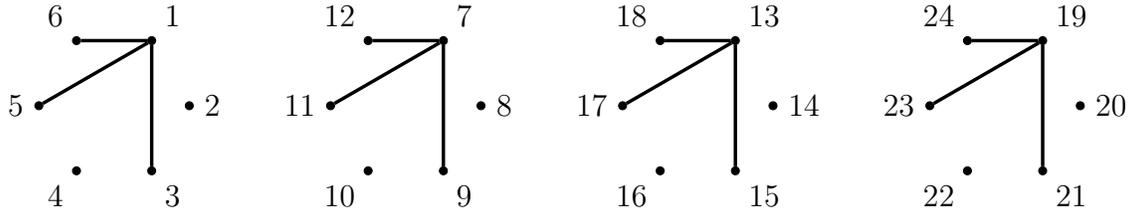
\begin{tikzpicture}[x=1cm, y=1cm]
 \tikzstyle{vertex}=[circle, draw, inner sep=0pt, minimum size=3pt]	
    \vertex (c1) at (60:1) [label=60:${1}$, fill=black]{};
 	\vertex (c6) at (120:1) [label=120:${6}$, fill=black]{};
        \vertex (c5) at (180:1) [label=180:${5}$, fill=black]{};
       
        \vertex (c4) at (240:1) [label=240:${4}$, fill=black]{};
        \vertex (c3) at (300:1) [label=300:${3}$, fill=black]{};
        \vertex (c2) at (360:1) [label=360:${2}$, fill=black]{};
		
 \path [line width=1.3pt]  
	(c1) edge (c3)
	(c1) edge (c5)
	(c1) edge (c6);	
\end{tikzpicture} \hspace{0.3cm}
\begin{tikzpicture}[x=1cm, y=1cm]
 \tikzstyle{vertex}=[circle, draw, inner sep=0pt, minimum size=3pt]	
    \vertex (c1) at (60:1) [label=60:${7}$, fill=black]{};
 	\vertex (c6) at (120:1) [label=120:${12}$, fill=black]{};
        \vertex (c5) at (180:1) [label=180:${11}$, fill=black]{};
        \vertex (c4) at (240:1) [label=240:${10}$, fill=black]{};
        \vertex (c3) at (300:1) [label=300:${9}$, fill=black]{};
        \vertex (c2) at (360:1) [label=360:${8}$, fill=black]{};
		
 \path [line width=1.3pt]  
	(c1) edge (c3)
	(c1) edge (c5)
	(c1) edge (c6);	
\end{tikzpicture} \hspace{0.3cm}
\begin{tikzpicture}[x=1cm, y=1cm]
 \tikzstyle{vertex}=[circle, draw, inner sep=0pt, minimum size=3pt]	
    \vertex (c1) at (60:1) [label=60:${13}$, fill=black]{};
 	\vertex (c6) at (120:1) [label=120:${18}$, fill=black]{};
        \vertex (c5) at (180:1) [label=180:${17}$, fill=black]{};
        \vertex (c4) at (240:1) [label=240:${16}$, fill=black]{};
        \vertex (c3) at (300:1) [label=300:${15}$, fill=black]{};
        \vertex (c2) at (360:1) [label=360:${14}$, fill=black]{};
		
 \path [line width=1.3pt]  
	(c1) edge (c3)
	(c1) edge (c5)
	(c1) edge (c6);	
\end{tikzpicture} \hspace{0.3cm}
\begin{tikzpicture}[x=1cm, y=1cm]
 \tikzstyle{vertex}=[circle, draw, inner sep=0pt, minimum size=3pt]	
    \vertex (c1) at (60:1) [label=60:${19}$, fill=black]{};
 	\vertex (c6) at (120:1) [label=120:${24}$, fill=black]{};
        \vertex (c5) at (180:1) [label=180:${23}$, fill=black]{};
        \vertex (c4) at (240:1) [label=240:${22}$, fill=black]{};
        \vertex (c3) at (300:1) [label=300:${21}$, fill=black]{};
        \vertex (c2) at (360:1) [label=360:${20}$, fill=black]{};
		
 \path [line width=1.3pt]  
	(c1) edge (c3)
	(c1) edge (c5)
	(c1) edge (c6);	
\end{tikzpicture}

Then $(V,\mathcal B)$, where $\mathcal B=\big(\bigcup\limits_{i=1}^{4}\mathcal B_i\big) \cup \big(\bigcup\limits_{i=1}^{3}\mathcal C_i^1\big) \cup \ldots \cup \big(\bigcup\limits_{i=1}^{3}\mathcal C_i^6\big)$ is a $23$-block-colourable $3$-star system of order $24$ with colour classes $\mathcal D^1,\ldots,\mathcal D^5, \mathcal C_1^1,\ldots,\mathcal C_1^6, \mathcal C_2^1,\ldots,\mathcal C_2^6,  \mathcal C_3^1,\ldots,\mathcal C_3^6$.
\end{example}

Next, we will consider the case where $n\equiv 1$ (mod $4e$).  Here, we will make use of a block colouring of an $e$-star system of order $n-1\equiv 0$ (mod $4e$) as obtained in Theorem~\ref{cong0_2e_even}, and extend it to obtain a system of order $n\equiv 1$ (mod $4e$) by adding an additional point.

\begin{thm} \label{cong1_even}
For $n\equiv 1$ (mod $4e$), there exists an $n$-block-colourable $e$-star system of order $n$.
\end{thm}

\begin{proof}
Let $n=4et+1$, where $t\geq 1$. Let $V=V'\cup\{x\}$ be the set of points where $V'=\{v_1^1,\ldots,v_{2e}^1,v_1^2,\ldots,v_{2e}^2,$ $\ldots,v_1^{2t},\ldots,v_{2e}^{2t}\}$.  Partition $V'$ into $2t$ subsets $V_1=\{v_1^1,\ldots,v_{2e}^1\},\ldots,\linebreak V_{2t}=\{v_1^{2t},\ldots,v_{2e}^{2t}\}$ of size $2e$. Let $(V',\mathcal B')$ be the $(4et-1)$-block-colourable $e$-star system of order $4et$ constructed in the proof of Theorem~\ref{cong0_2e_even}, with colour classes $\mathcal D^1, \ldots, \mathcal D^{2e-1}$, $\mathcal C_1^1,\ldots, \mathcal C_1^{2e},$ $\ldots$, $\mathcal C_{2t-1}^1,\ldots, \mathcal C_{2t-1}^{2e}$.

The next step is to decompose the edges between the point $x$ and the subsets $V_1,\ldots,V_{2t}$ into $e$-stars and colour these $e$-stars. For each $i$ where $1\leq i\leq 2t-1$, we decompose the edges between the point $x$ and the subset $V_i$ into $e$-stars $\{x;v_1^i,\ldots,v_e^i\}$ and $\{x;v_{e+1}^i,\ldots,v_{2e}^i\}$. 
To colour these, we reuse the colours associated with the $1$-factor $F_{i}$ from the proof of Theorem~\ref{cong0_2e_even}. 
In the $1$-factor $F_{i}$, take $V_i$ to be the first element of the ordered pair $(V_i,V_j)\in F_{i}$.  With this choice, we can see from the construction in the proof of Theorem~\ref{cong0_2e_even} that the $e$-stars $\{x;v_1^i,\ldots,v_{e}^i\}$ and $\{x;v_{e+1}^i,\ldots,v_{2e}^i\}$ do not have any intersection with the points in the colour classes $C_{i}^{2e}$ and $C_{i}^1$ respectively. Therefore, for each $i$ where $1\leq i\leq 2t-1$, we let $(\mathcal C_i^1)'=\mathcal C_i^1 \cup \big\{\{x;v_{e+1}^i,\ldots,v_{2e}^i\}\big\}$ and $(\mathcal C_i^{2e})'=\mathcal C_i^{2e} \cup\big\{ \{x;v_1^i,\ldots,v_{e}^i\}\big\}$. 

It only remains to decompose the edges between the point $x$ and the subset $V_{2t}$ into $e$-stars and colour these. We decompose these edges into $e$-stars $\{x;v_1^{2t},\ldots,v_e^{2t}\}$ and $\{x;v_{e+1}^{2t},\ldots,v_{2e}^{2t}\}$. To colour these $e$-stars, we define two new colour classes $\mathcal A_1=\big\{\{x;v_1^{2t},\ldots,v_e^{2t}\}\big\}$ and $\mathcal A_2=\big\{\{x;v_{e+1}^{2t},\ldots,v_{2e}^{2t}\}\big\}$.

Therefore $(V,\mathcal B)$ where $\mathcal B=\mathcal B' \cup \big(\bigcup\limits_{i=1}^{2t} \big\{\{x;v_1^i,\ldots,v_e^i\}, \{x;v_{e+1}^i,\ldots,v_{2e}^i\}\big\}\big)$ is a $(4et+1)$-block-colourable $e$-star system of order $4et+1$, with colour classes $\mathcal D^1, \ldots, \mathcal D^{2e-1}$, $(\mathcal C_1^1)',\mathcal C_1^2,\ldots,\linebreak \mathcal C_1^{2e-1}, (\mathcal C_1^{2e})', \ldots, (\mathcal C_{2t-1}^1)', \mathcal C_{2t-1}^2, \ldots, \mathcal C_{2t-1}^{2e-1}, (\mathcal C_{2t-1}^{2e})',\mathcal A_1, \mathcal A_2$.
\end{proof}

Our next result considers the case where $n\equiv 2e$ (mod $4e$).  Here, we will adapt the proof of Theorem~\ref{cong0_2e_even}, but will use a near $1$-factorization instead of a $1$-factorization.

\begin{thm}\label{cong0_odd}
For $n\equiv 2e$ (mod $4e$), there exists an $n$-block-colourable $e$-star system of order $n$, except for $n=2e$ where all $e$-star systems of order $2e$ are trivially $(n-1)$-block-chromatic.
\end{thm}

\begin{proof}
Since the case $n=2e$ is trivial, we will suppose that $n=4et+2e$ where $t\geq 1$. Let $V=\{v_1^1,\ldots,v_{2e}^1,v_1^2,\ldots,v_{2e}^2,\ldots,v_1^{2t+1},\ldots,v_{2e}^{2t+1}\}$ be the set of points. Partition $V$ into $2t+1$ subsets $V_1=\{v_1^1,\ldots,v_{2e}^1\},\ldots,V_{2t+1}=\{v_1^{2t+1},\ldots,v_{2e}^{2t+1}\}$ of size $2e$. For each $i$ where $1\leq i\leq 2t+1$, take an $e$-star system $(V_i,\mathcal B_i)$ of order $2e$, with blocks $\mathcal B_i=\{B_i^1,\ldots,B_i^{2e-1}\}$, which is necessarily $(2e-1)$-block-chromatic. 

Next, obtain a complete graph $K_{2t+1}$ whose vertices are $V_1,\ldots,V_{2t+1}$. Since $2t+1$ is odd, $K_{2t+1}$ admits a near 1-factorization; hence we can partition the set of all pairs of subsets chosen from $V_1,\ldots,V_{2t+1}$ into $2t+1$ near $1$-factors $F_1,\ldots,F_{2t+1}$. For each $i$ where $1\leq i\leq 2t+1$, we let $\mathcal C_i^1,\ldots,\mathcal C_i^{2e}$ be colour classes defined in a similar manner as those in the proof of Theorem~\ref{cong0_2e_even}. Also, for each $i$, we suppose that $V_i$ is the missing point of the near 1-factor $F_i$. 

To colour the $e$-stars in each $\mathcal B_i$, we can reuse colours associated with the near $1$-factor $F_i$. So we let $(\mathcal C_i^1)'=\mathcal C_i^1 \cup \{B_i^1\},\ldots, (\mathcal C_i^{2e-1})'=\mathcal C_i^{2e-1} \cup \{B_i^{2e-1}\}$.

Then $(V,\mathcal B)$, where $\mathcal B=\big(\bigcup\limits_{i=1}^{2t+1} \mathcal B_i\big) \cup \big(\bigcup_{i=1}^{2t+1} \mathcal C_i^1\big) \cup 
\ldots \cup \big(\bigcup_{i=1}^{2t+1} \mathcal C_i^{2e}\big)$ is a $(4et+2e)$-block-colourable $e$-star system of order $4et+2e$ with colour classes 
$(\mathcal C_1^1)', \ldots, (\mathcal C_1^{2e-1})', \mathcal C_1^{2e}, \ldots,\linebreak (\mathcal C_{2t+1}^1)',\ldots, (\mathcal C_{2t+1}^{2e-1})', \mathcal C_{2t+1}^{2e}$.
\end{proof}

The final result of this section considers the case where $n\equiv 2e+1$ (mod $4e$).  Similar to Theorem~\ref{cong1_even}, we will extend an $e$-star system of order $n-1\equiv 2e$ (mod $4e$) by introducing a new point, although the details are not identical.

\begin{thm} \label{cong1_odd}
For $n\equiv 2e+1$ (mod $4e$), there exists an $(n-1)$-block-colourable $e$-star system of order $n$, except for $n=2e+1$ where all $e$-star systems of order $2e+1$ are trivially $n$-block-chromatic.
\end{thm}

\begin{proof}
Since the case $n=2e+1$ is trivial, we can suppose that $n=4et+2e+1$ where $t\geq 1$. 

Let $V=V'\cup \{x\}$, where $V'=\{v_1^1,\ldots,v_{2e}^1,v_1^2,\ldots,v_{2e}^2,\ldots,v_1^{2t+1},\ldots,v_{2e}^{2t+1}\}$. Partition $V'$ into $2t+1$ subsets $V_1=\{v_1^1,\ldots,v_{2e}^1\},$ $V_2=\{v_1^2,\ldots,v_{2e}^2\},$ $\ldots,$ $V_{2t+1}=\{v_1^{2t+1},\ldots,v_{2e}^{2t+1}\}$ of size $2e$.
Let $(V',\mathcal B')$ be the $(4et+2e)$-block-colourable $e$-star system of order $4et+2e$ with colour classes $(\mathcal C_1^1)', \ldots, (\mathcal C_1^{2e-1})',\mathcal C_1^{2e}, \ldots, (\mathcal C_{2t+1}^1)', \ldots, (\mathcal C_{2t+1}^{2e-1})', \mathcal C_{2t+1}^{2e}$ constructed in the proof of Theorem~\ref{cong0_odd}. 

The next step is to decompose the edges between the point $x$ and the subsets $V_1,\ldots,V_{2t+1}$ into $e$-stars and colour these $e$-stars. For each $i$ where $1\leq i\leq 2t+1$, we decompose the edges between the point $x$ and the subset $V_i$ into $e$-stars $\{x;v_1^i,\ldots,v_e^i\}$ and $\{x;v_{e+1}^i,\ldots,v_{2e}^i\}$. To colour these, we will reuse the existing colours: these will be taken from the colour classes associated with the $1$-factor $F_{i-1}$ (where $F_0=F_{2t+1}$).  
In the $1$-factor $F_{i-1}$, suppose that $V_i$ is the first element in the ordered pair $(V_i,V_j)\in F_{i-1}$; we can see from the constructions in the proofs of Theorem~\ref{cong0_2e_even} and Theorem~\ref{cong0_odd} that the $e$-stars $\{x;v_1^i,\ldots,v_e^i\}$ and $\{x;v_{e+1}^i,\ldots,v_{2e}^i\}$ do not have any intersection with the points in the colour classes $C_{i-1}^{2e}$ and $(C_{i-1}^1)'$ respectively. Therefore, for each $i$ where $1\leq i\leq 2t+1$, we let $(\mathcal C_{i-1}^1)''= (\mathcal C_{i-1}^1)' \cup \big\{\{x;v_{e+1}^{i},\ldots,v_{2e}^{i}\}\big\}$ and $(\mathcal C_{i-1}^{2e})''= \mathcal C_{i-1}^{2e} \cup \big\{\{x;v_1^{i},\ldots,v_e^{i}\}\big\}$. 

Then $(V,\mathcal B)$, where $\mathcal B=\mathcal B' \cup \big(\bigcup\limits_{i=1}^{2t+1} \big\{\{x;v_1^i,\ldots,v_e^i\}, \{x;v_{e+1}^i,\ldots,v_{2e}^i\}\big\}\big)$, is a $(4et+2e)$-block-colourable $e$-star system of order $4et+2e+1$ with colour classes $(\mathcal C_1^1)'', (\mathcal C_1^2)', \ldots, (\mathcal C_1^{2e-1})', \linebreak (\mathcal C_1^{2e})'', \ldots, (\mathcal C_{2t+1}^1)'', (\mathcal C_{2t+1}^2)', \ldots, (\mathcal C_{2t+1}^{2e-1})', (\mathcal C_{2t+1}^{2e})''$.
\end{proof}

Combining all of the results of this section, we have the following corollary.

\begin{cor} \label{combined_e}
For all $e\geq 3$, and each $n\equiv 0,1$ (mod $2e$), there exists either an $(n-1)$-block-colourable or an $n$-block-colourable $e$-star system of order $n$.
\end{cor}

\section{An upper bound: the case of $3$-star systems} \label{sec:bound-3}

In general, the results of Section~\ref{sec:bound-e} do not cover all the possible congruence classes mod~$2e$ for which $e$-star systems can exist.  In particular, if $e=3$, they only cover the cases where $n\equiv 0,1$ mod~$6$; it is known that $3$-star systems will also exist when $n\equiv 3,4$ mod~$6$.  In this section, we will give constructions of block-coloured $3$-star systems for these additional congruence classes (considering them modulo $12$), which attain the same upper bounds on the minimum number of colours.  As an additional ingredient, we will make use of the $8$-block-colourable $S_3(9)$ given in Example~\ref{3star9-colour}.

\begin{thm}\label{cong3_even}
For $n\equiv 3$ (mod $12$), there exists an $(n-1)$-block-colourable $3$-star system of order $n$.
\end{thm}
\begin{proof}
Let $n=12t+3$. Since there is no $3$-star system of order $3$, we will assume that $t\geq 1$. Let $V$ be the set of points. We will partition $V$ into $2t$ subsets, $V_1=\{v_1^1,\ldots,v_6^1\}, \ldots, V_{2t-1}=\{v_1^{2t-1},\ldots,v_6^{2t-1}\}$, $V_{2t}=\{v_1^{2t},\ldots,v_9^{2t}\}$ (so that $V_{2t}$ has size $9$ and the others have size $6$).  We will place a $3$-star system on each of these: for $1\leq i\leq 2t-1$, we let $\mathcal B_i=\big\{\{v_1^i;v_3^i,v_5^i,v_6^i\}, \{v_2^i;v_1^i,v_3^i,v_6^i\}, \{v_4^i;v_1^i,v_2^i,v_3^i\}, \{v_5^i;v_2^i,v_3^i,v_4^i\}, \{v_6^i;v_3^i,v_4^i,v_5^i\}\big\}$, so that $(V_i,\mathcal B_i)$ is a copy of the $S_3(6)$ in Example~\ref{3star6}, which is necessarily $5$-block-chromatic.
On $V_{2t}$, we will place a copy of the $S_3(9)$ from Example~\ref{3star9-colour}, which is $8$-block-chromatic: we let 
$\mathcal B_{2t}=\big\{\{v_1^{2t};v_3^{2t},v_5^{2t},v_6^{2t}\}, 
\{v_2^{2t};v_1^{2t},v_3^{2t},v_6^{2t}\}, 
\{v_4^{2t};v_1^{2t},v_2^{2t},v_3^{2t}\}, 
\{v_5^{2t};v_2^{2t},v_3^{2t},v_4^{2t}\}, 
\{v_6^{2t};v_3^{2t},v_4^{2t},v_5^{2t}\}, \linebreak
\{v_7^{2t};v_1^{2t},v_2^{2t},v_3^{2t}\}, 
\{v_8^{2t},v_4^{2t},v_5^{2t},v_9^{2t}\}, 
\{v_7^{2t},v_4^{2t},v_5^{2t},v_8^{2t}\}, 
\{v_8^{2t};v_1^{2t},v_2^{2t},v_3^{2t}\}, 
\{v_6^{2t},v_7^{2t},v_8^{2t},v_9^{2t}\}, \linebreak
\{v_9^{2t},v_1^{2t},v_2^{2t},v_3^{2t}\},
\{v_9^{2t},v_4^{2t},v_5^{2t},v_7^{2t}\}\big\}$.  Note that this has a subsystem isomorphic to the $S_3(6)$ placed on each of $V_1,\ldots,V_{2t-1}$.

We now construct a complete graph $K_{2t}$ in which the vertices are $V_1,\ldots,V_{2t}$. Since $2t$ is even, $K_{2t}$ admits a 1-factorization, so we can partition the set of all pairs of subsets from $V_1,\ldots,V_{2t}$ into $2t-1$ $1$-factors $F_1,\ldots,F_{2t-1}$. Without loss of generality, we assume that $F_1=\big\{(V_1,V_2),(V_3,V_4),\ldots,(V_{2t-1},V_{2t})\big\}$. Now, for each $j$ where $1\leq j\leq t-1$, let 
\renewcommand{\arraystretch}{1.2}%
\[ \begin{array}{l}
\mathcal F_{2j-1}^1=\big\{ \{v_1^{2j-1};v_1^{2j},v_2^{2j},v_3^{2j}\}, \{v_2^{2j-1};v_4^{2j},v_5^{2j},v_6^{2j}\}\big\},\\ 
\mathcal F_{2j-1}^2=\big\{\{v_1^{2j-1};v_4^{2j},v_5^{2j},v_6^{2j}\}, \{v_2^{2j-1};v_1^{2j},v_2^{2j},v_3^{2j}\}\big\}, \\
\mathcal F_{2j-1}^3=\big\{\{v_3^{2j-1},v_1^{2j},v_2^{2j},v_3^{2j}\}, \{v_4^{2j-1},v_4^{2j},v_5^{2j},v_6^{2j}\}\big\}, \\
\mathcal F_{2j-1}^4=\big\{\{v_3^{2j-1};v_4^{2j},v_5^{2j},v_6^{2j}\}, \{v_4^{2j-1};v_1^{2i},v_2^{2j},v_3^{2j}\}\big\}, \\
\mathcal F_{2j-1}^5=\big\{\{v_5^{2j-1};v_1^{2j},v_2^{2j},v_3^{2j}\}, \{v_6^{2j-1};v_4^{2j},v_5^{2j},v_6^{2j}\}\big\}, \\
\mathcal F_{2j-1}^6=\big\{\{v_5^{2j-1};v_4^{2j},v_5^{2j},v_6^{2j}\}, \{v_6^{2j-1};v_1^{2j},v_2^{2j},v_3^{2j}\}\big\}.
\end{array} \]%
\renewcommand{\arraystretch}{1.0}%
Moreover, let
\renewcommand{\arraystretch}{1.2}%
\[ \begin{array}{l}
\mathcal F_{2t-1}^1=\big\{ \{v_1^{2t-1};v_1^{2t},v_2^{2t},v_3^{2t}\}, \{v_2^{2t-1};v_4^{2t},v_5^{2t},v_6^{2t}\}, \{v_3^{2t-1};v_7^{2t},v_8^{2t},v_9^{2t}\}\big\},\\
\mathcal F_{2t-1}^2=\big\{\{v_1^{2t-1};v_4^{2t},v_5^{2t},v_6^{2t}\}, \{v_2^{2t-1};v_7^{2t},v_8^{2t},v_9^{2t}\},\{v_3^{2t-1};v_1^{2t},v_2^{2t},v_3^{2t}\}\big\}, \\
\mathcal F_{2t-1}^3=\big\{\{v_1^{2t-1};v_7^{2t},v_8^{2t},v_9^{2t}\}, \{v_2^{2t-1},v_1^{2t},v_2^{2t},v_3^{2t}\}, \{v_3^{2t-1};v_4^{2t},v_5^{2t},v_6^{2t}\}\big\}, \\
\mathcal F_{2t-1}^4=\big\{ \{v_4^{2t-1};v_1^{2t},v_2^{2t},v_3^{2t}\}, \{v_5^{2t-1};v_4^{2t},v_5^{2t},v_6^{2t}\},\{v_6^{2t-1};v_7^{2t},v_8^{2t},v_9^{2t}\}\big\}, \\
\mathcal F_{2t-1}^5=\big\{\{v_4^{2t-1};v_4^{2t},v_5^{2t},v_6^{2t}\}, \{v_5^{2t-1};v_7^{2t},v_8^{2t},v_9^{2t}\},\{v_6^{2t-1};v_1^{2t},v_2^{2t},v_3^{2t}\}\big\}, \\
\mathcal F_{2t-1}^6=\big\{\{v_4^{2t-1};v_7^{2t},v_8^{2t},v_9^{2t}\}, \{v_5^{2t-1},v_1^{2t},v_2^{2t},v_3^{2t}\}, \{v_6^{2t-1};v_4^{2t},v_5^{2t},v_6^{2t}\}\big\}.
\end{array} \]%
\renewcommand{\arraystretch}{1.0}%
Then the edges between pairs of subsets in the 1-factor $F_1$ can be partitioned into six colour classes $\mathcal C_1^1=\bigcup\limits_{j=1}^{t} \mathcal F_{2j-1}^1, \ldots, \mathcal C_1^6=\bigcup\limits_{j=1}^{t} \mathcal F_{2j-1}^6$. For each $i$ where $2\leq i\leq 2t-1$, the edges between pairs of subsets in the $1$-factor $F_i$ can be partitioned into six colour classes $\mathcal C_i^1, \ldots, \mathcal C_i^6$ in a similar manner.  Together, these use a total of $6(2t-1)$ colours.

It now remains to colour the $3$-stars in each $\mathcal B_i$ (where $1\leq i\leq 2t$).  Since these are formed of disjoint $3$-star systems which are either $5$-block-chromatic (for $1\leq i\leq 2t-1$) or $8$-block-chromatic (for $i=2t$), we require eight colours.  Using the $8$-block colouring given in Example~\ref{3star9-colour}, we obtain the following colour classes:%
\renewcommand{\arraystretch}{1.2}%
\[ \begin{array}{l}
\mathcal D^1=\bigcup\limits_{i=1}^{2t}\big\{\{v_1^i;v_3^i,v_5^i,v_6^i\}\big\} \\
\mathcal D^2=\bigcup\limits_{i=1}^{2t}\big\{\{v_2^i;v_1^i,v_3^i,v_6^i\}\big\} \cup \big\{\{v_9^{2t};v_{4}^{2t},v_5^{2t},v_7^{2t}\}\big\} \\
\mathcal D^3=\bigcup\limits_{i=1}^{2t}\big\{\{v_4^i;v_1^i,v_2^i,v_3^i\}\big\} \cup \big\{\{v_6^{2t};v_7^{2t},v_8^{2t},v_9^{2t}\}\big\} \\
\mathcal D^4=\bigcup\limits_{i=1}^{2t}\big\{\{v_5^i;v_2^i,v_3^i,v_4^i\}\big\} 
\end{array}
\quad
\begin{array}{l}
\mathcal D^5=\bigcup\limits_{i=1}^{2t}\big\{\{v_6^i;v_3^i,v_4^i,v_5^i\}\big\} \\
\mathcal D^6=\big\{\{v_7^{2t};v_1^{2t},v_2^{2t},v_3^{2t}\}, \{v_8^{2t},v_4^{2t},v_5^{2t},v_9^{2t}\}\big\} \\
\mathcal D^7=\big\{\{v_7^{2t};v_4^{2t},v_5^{2t},v_8^{2t}\}, \{v_9^{2t};v_1^{2t},v_2^{2t},v_3^{2t}\}\big\} \\
\mathcal D^8=\big\{\{v_8^{2t};v_1^{2t},v_2^{2t},v_3^{2t}\}\big\}
\end{array} \]%
\renewcommand{\arraystretch}{1.0}%
Together, these $6(2t-1)+8 = 12t+2 =n-1$ colour classes yield a $(12t+2)$-block-colourable $3$-star system of order $12t+3$, with colour classes $\mathcal C_1^1, \ldots, \mathcal C_1^6, \ldots, \mathcal C_{2t-1}^1, \ldots, \mathcal C_{2t-1}^6, \mathcal D^1, \ldots, \mathcal D^8$.
\end{proof}

Similar to Theorem~\ref{cong1_even}, we may consider the case where $n\equiv 4$ (mod $12$) by taking our previous construction and extend it by adding a new point.

\begin{thm} \label{cong4_even}
For $n\equiv 4$ (mod $12$), there exists an $n$-block-colourable $3$-star system of order $n$. 
\end{thm}
\begin{proof}
Let $n=12t+4$. Since there is no $3$-star system of order $4$, we will assume that $t\geq 1$. Here, we will extend the block colouring of a $3$-star system of order $n-1\equiv 3$ (mod $12$) as obtained in the proof of Theorem~\ref{cong3_even}.

Let $V=V' \cup \{x\}$ be the set of points where $|V'|=12t+3$. Partition $V'$ into $2t-1$ subsets of size~$6$, namely $V_1=\{v_1^1,\ldots,v_6^1\}, \ldots, V_{2t-1}=\{v_1^{2t-1},\ldots,v_6^{2t-1}\}$, and one of size~$9$, namely $V_{2t}=\{v_1^{2t},\ldots,v_9^{2t}\}$. Let $\mathcal (V',\mathcal B')$ be a $(12t+2)$-block-colourable 3-star system of order $12t+3$ as in the proof of Theorem~\ref{cong3_even}, with colour classes $\mathcal C_1^1, \ldots, \mathcal C_1^6, \ldots, \mathcal C_{2t-1}^1, \ldots, \mathcal C_{2t-1}^6, \mathcal D^1, \ldots, \mathcal D^8$.

The next step is to decompose the edges between the point $x$ and the subsets $V_1,\ldots,V_{2t}$ into $3$-stars and then to colour them. First, for each $i$ where $1\leq i\leq 2t-1$, we decompose the edges between the point $x$ and the subset $V_i$ into two $3$-stars $\{x;v_1^i,v_2^i,v_3^i\}$ and $\{x;v_4^i,v_5^i,v_6^i\}$. To colour these, we reuse the colours of $\mathcal C_i^6$ and $\mathcal C_i^1$ associated to the $1$-factor $F_{i}$ in the proof of Theorem~\ref{cong3_even}. So, for each $i$ where $1\leq i\leq 2t-1$, we let $(\mathcal C_i^1)'= \mathcal C_i^1 \cup \big\{\{x;v_4^i,v_5^i,v_6^i\}\big\}$ and $(\mathcal C_i^6)'= \mathcal C_i^6 \cup \big\{\{x;v_1^i,v_2^i,v_3^i\}\big\}$.

It only remains to decompose the edges between the point $x$ and the subset $V_{2t}$ into $3$-stars and colour them. This time, we obtain three $3$-stars 
$\{x;v_1^{2t},v_2^{2t},v_3^{2t}\}$, $\{x;v_4^{2t},v_5^{2t},v_6^{2t}\}$, and $\{x;v_7^{2t},v_8^{2t},v_9^{2t}\}$. To colour these, we define two new colour classes, namely $\mathcal E^1 =\big\{\{x;v_1^{2t},v_2^{2t},v_3^{2t}\}\big\}$ and $\mathcal E^2 =\big\{\{x;v_4^{2t},v_5^{2t},v_6^{2t}\}\big\}$, while the third block can be assigned the same colour as class $\mathcal D^1$.  So we let $(\mathcal D^1)' = \mathcal D^1 \cup \big\{\{x;v_7^{2t},v_8^{2t},v_9^{2t}\}\big\}$.

Pulling all of this together, we obtain a $(12t+4)$-block-colourable 3-star system of order $12t+4$, with colour classes $(\mathcal C_1^1)', \mathcal C_1^2, \ldots, \mathcal C_1^5$, $(\mathcal C_1^6)', \ldots, (\mathcal C_{2t-1}^1)', \mathcal C_{2t-1}^2, \ldots, \mathcal C_{2t-1}^5, (\mathcal C_{2t-1}^6)', (\mathcal D^1)', \linebreak \mathcal D^2, \ldots,\mathcal D^8, \mathcal E^1, \mathcal E^2$.
\end{proof}

Our next result uses a combination of the approaches of Theorem~\ref{cong1_odd} and Theorem~\ref{cong3_even}, to consider the case where $n\equiv 9$ (mod $12$), using a near $1$-factorization.

\begin{thm}\label{cong3_odd}
For $n\equiv 9$ (mod $12$), there exists an $(n-1)$-block-colourable $3$-star system of order $n$.
\end{thm}

\begin{proof}
Let $n=12t+9$.  If $t=0$, we have an $8$-block-colourable $S_3(9)$ given in Example~\ref{3star9-colour}. In what follows, we will assume that $t\geq 1$.

Let $V$ be the set of points. Partition $V$ into $2t$ subsets of size $6$, namely $V_1=\{v_1^1,\ldots,v_6^1\}, \linebreak \ldots, V_{2t}=\{v_1^{2t},\ldots,v_6^{2t}\}$, and one subset of size $9$, namely $V_{2t+1}=\{v_1^{2t+1},\ldots,v_9^{2t+1}\}$. On each $V_i$ for $1\leq i\leq 2t$, we will place an $S_3(6)$ (from Example~\ref{3star6}, which is $5$-block-chromatic) with blocks $\mathcal B_i$ as in the proof of Theorem~\ref{cong3_even}. On $V_{2t+1}$ we place a copy of the $8$-block chromatic $S_3(9)$ from Example~\ref{3star9-colour}; we label its blocks $\mathcal B_{2t+1}$ as in the proof of Theorem~\ref{cong3_even}, but with the superscript $2t+1$.

Once again, we construct a complete graph $K_{2t+1}$ in which the vertices are $V_1,\ldots,V_{2t+1}$. Since $2t+1$ is odd, this $K_{2t+1}$ admits a near 1-factorization, and hence we can partition the set of all pairs of the subsets $V_1,\ldots,V_{2t+1}$ into $2t+1$ near $1$-factors, $F_1,\ldots,F_{2t+1}$. Suppose that $V_1,\ldots,V_{2t+1}$ are the missing points of the $1$-factors $F_1,\ldots,F_{2t+1}$ respectively. Assume that $F_1=\big\{(V_2,V_3),(V_4,V_5),\ldots,(V_{2t},V_{2t+1})\big\}$. For each $j$ where $1\leq j\leq t-1$, let 
\renewcommand{\arraystretch}{1.2}
\[ \begin{array}{l}
\mathcal F_{2j}^1=\big\{ \{v_1^{2j};v_1^{2j+1},v_2^{2j+1},v_3^{2j+1}\}, \{v_2^{2j};v_4^{2j+1},v_5^{2j+1},v_6^{2j+1}\}\big\},\\ 
\mathcal F_{2j}^2=\big\{\{v_1^{2j};v_4^{2j+1},v_5^{2j+1},v_6^{2j+1}\}, \{v_2^{2j};v_1^{2j+1},v_2^{2j+1},v_3^{2j+1}\}\big\}, \\
\mathcal F_{2j}^3=\big\{\{v_3^{2j},v_1^{2j+1},v_2^{2j+1},v_3^{2j+1}\}, \{v_4^{2j},v_4^{2j+1},v_5^{2j+1},v_6^{2j+1}\}\big\},\\
\mathcal F_{2j}^4=\big\{\{v_3^{2j};v_4^{2j+1},v_5^{2j+1},v_6^{2j+1}\}, \{v_4^{2j};v_1^{2j+1},v_2^{2j+1},v_3^{2j+1}\}\big\},\\
\mathcal F_{2j}^5=\big\{\{v_5^{2j};v_1^{2j+1},v_2^{2j+1},v_3^{2j+1}\}, \{v_6^{2j};v_4^{2j+1},v_5^{2j+1},v_6^{2j+1}\}\big\},\\
\mathcal F_{2j}^6=\big\{\{v_5^{2j};v_4^{2j+1},v_5^{2j+1},v_6^{2j+1}\}, \{v_6^{2j};v_1^{2j+1},v_2^{2j+1},v_3^{2j+1}\}\big\}. 
\end{array} \]
\renewcommand{\arraystretch}{1.0}
Also, let
\renewcommand{\arraystretch}{1.2}
\[ \begin{array}{l}
\mathcal F_{2t}^1=\big\{ \{v_1^{2t};v_1^{2t+1},v_2^{2t+1},v_3^{2t+1}\}, \{v_2^{2t};v_4^{2t+1},v_5^{2t+1},v_6^{2t+1}\}, \{v_3^{2t};v_7^{2t+1},v_8^{2t+1},v_9^{2t+1}\}\big\},\\
\mathcal F_{2t}^2=\big\{\{v_1^{2t};v_4^{2t+1}, v_5^{2t+1},v_6^{2t+1}\}, \{v_2^{2t};v_7^{2t+1},v_8^{2t+1},v_9^{2t+1}\},\{v_3^{2t};v_1^{2t+1},v_2^{2t+1},v_3^{2t+1}\}\big\},\\
\mathcal F_{2t}^3=\big\{\{v_1^{2t};v_7^{2t+1},v_8^{2t+1},v_9^{2t+1}\}, \{v_2^{2t};v_1^{2t+1},v_2^{2t+1},v_3^{2t+1}\}, \{v_3^{2t};v_4^{2t+1},v_5^{2t+1},v_6^{2t+1}\}\big\},\\
\mathcal F_{2t}^4=\big\{ \{v_4^{2t};v_1^{2t+1},v_2^{2t+1},v_3^{2t+1}\}, \{v_5^{2t};v_4^{2t+1},v_5^{2t+1},v_6^{2t+1}\}, \{v_6^{2t};v_7^{2t+1},v_8^{2t+1},v_9^{2t+1}\}\big\},\\
\mathcal F_{2t}^5=\big\{\{v_4^{2t};v_4^{2t+1},v_5^{2t+1},v_6^{2t+1}\}, \{v_5^{2t};v_7^{2t+1},v_8^{2t+1},v_9^{2t+1}\},\{v_6^{2t};v_1^{2t+1},v_2^{2t+1},v_3^{2t+1}\}\big\},\\ 
\mathcal F_{2t}^6=\big\{\{v_4^{2t};v_7^{2t+1},v_8^{2t+1},v_9^{2t+1}\}, \{v_5^{2t};v_1^{2t+1},v_2^{2t+1},v_3^{2t+1}\}, \{v_6^{2t};v_4^{2t+1},v_5^{2t+1},v_6^{2t+1}\}\big\}.
\end{array} \]%
\renewcommand{\arraystretch}{1.0}%
Then the edges between pairs of subsets in the $1$-factor $F_1$ can be partitioned into six colour classes $\mathcal C_1^1=\bigcup\limits_{j=1}^{t} \mathcal F_{2j}^1, \ldots, \mathcal C_1^6=\bigcup\limits_{j=1}^{t} \mathcal F_{2j}^6$. For each $i$ where $2\leq i\leq 2t+1$, the edges between pairs of subsets in the $1$-factor $F_i$ can be partitioned into six colour classes $\mathcal C_i^1, \ldots, \mathcal C_i^6$ in a similar manner. 

Next, for each $i$ where $1\leq i\leq 2t$, we will colour the five $3$-stars in $\mathcal B_i$. Since $V_i$ is the missing point of the $1$-factor $F_i$, we can reuse the colours from classes $\mathcal C_i^1$, $\mathcal C_i^2$, $\mathcal C_i^3$, $\mathcal C_i^4$ and $\mathcal C_i^5$ associated with the $1$-factor $F_i$. So, for each $i$ where $1\leq i\leq 2t$, we let $(\mathcal C_i^1)'=\mathcal C_i^1 \cup \big\{\{v_1^i;v_3^i,v_5^i,v_6^i\}\big\}$, $(\mathcal C_i^2)'=\mathcal C_i^2 \cup \big\{\{v_2^i;v_1^i,v_3^i,v_6^i\}\big\}$, $(\mathcal C_i^3)'=\mathcal C_i^3 \cup \big\{\{v_4^i;v_1^i,v_2^i,v_3^i\}\big\}$, $(\mathcal C_i^4)'=\mathcal C_i^4 \cup \big\{\{v_5^i;v_2^i,v_3^i,v_4^i\}\big\}$, and $(\mathcal C_i^5)'=\mathcal C_i^5 \cup \big\{\{v_6^i;v_3^i,v_4^i,v_5^i\}\big\}$. 

It remains to colour the $3$-stars in $\mathcal B_{2t+1}$, which requires eight colours.  Since $V_{2t+1}$ is the missing point of the $1$-factor $F_{2t+1}$, we can reuse the six colours from classes $\mathcal C_{2t+1}^1, \ldots, \mathcal C_{2t+1}^6$ associated with the $1$-factor $F_{2t+1}$, as follows:
\renewcommand{\arraystretch}{1.2}
\[ \begin{array}{l}
(\mathcal C_{2t+1}^1)'=\mathcal C_{2t+1}^1 \cup \big\{\{v_1^{2t+1};v_3^{2t+1},v_5^{2t+1},v_6^{2t+1}\}\big\}, \\
(\mathcal C_{2t+1}^2)'=\mathcal C_{2t+1}^2 \cup \big\{\{v_2^{2t+1};v_1^{2t+1},v_3^{2t+1},v_6^{2t+1}\}, \{v_9^{2t+1};v_{4}^{2t+1},v_5^{2t+1},v_7^{2t+1}\}\big\}, \\
(\mathcal C_{2t+1}^3)'=\mathcal C_{2t+1}^3 \cup \big\{\{v_4^{2t+1};v_1^{2t+1},v_2^{2t+1},v_3^{2t+1}\}, \{v_6^{2t+1};v_7^{2t+1},v_8^{2t+1},v_9^{2t+1}\}\big\}, \\
(\mathcal C_{2t+1}^4)'=\mathcal C_{2t+1}^4 \cup \big\{\{v_5^{2t+1};v_2^{2t+1},v_3^{2t+1},v_4^{2t+1}\}\big\}, \\
(\mathcal C_{2t+1}^5)'=\mathcal C_{2t+1}^5 \cup \big\{\{v_6^{2t+1};v_3^{2t+1},v_4^{2t+1},v_5^{2t+1}\}\big\}, \\
(\mathcal C_{2t+1}^6)'=\mathcal C_{2t+1}^6 \cup \big\{\{v_7^{2t+1};v_1^{2t+1},v_2^{2t+1},v_3^{2t+1}\}, \{v_8^{2t+1},v_4^{2t+1},v_5^{2t+1},v_9^{2t+1}\}\big\}.
\end{array} \]%
\renewcommand{\arraystretch}{1.0}%
For the remaining blocks, we introduce two new colours:
\renewcommand{\arraystretch}{1.2}%
\[ \begin{array}{l}
\mathcal A^1=\big\{\{v_7^{2t+1};v_4^{2t+1},v_5^{2t+1},v_8^{2t+1}\}, \{v_9^{2t+1};v_1^{2t+1},v_2^{2t+1},v_3^{2t+1}\}\big\}, \\
\mathcal A^2=\big\{\{v_8^{2t+1};v_1^{2t+1},v_2^{2t+1},v_3^{2t+1}\}\big\}
\end{array} \]%
\renewcommand{\arraystretch}{1.0}%
Together, these yield a $(12t+8)$-block-colourable $3$-star system of order $12t+9$ with colour classes $(\mathcal C_1^1)', \ldots (\mathcal C_1^5)', \mathcal C_1^6, \ldots, (\mathcal C_{2t}^1 )', \ldots, (\mathcal C_{2t}^5)', \mathcal C_{2t}^6, (\mathcal C_{2t+1}^1 )', \ldots, (\mathcal C_{2t+1}^6)', \mathcal A^1, \mathcal A^2$.
\end{proof}

Our last main result is another ``extension'' result, which enables us to consider the final congruence class for $3$-star systems.

\begin{thm} \label{cong4_odd}
For $n\equiv 10$ (mod $12$), there exists an $(n-1)$-block-colourable $3$-star system of order $n$.
\end{thm}

\begin{proof}
We remark that $9$-block colourable $3$-star systems of order $10$ were obtained by computer search in subsection~\ref{sec:soicher} (see Table~\ref{table:chromatic}), so it suffices to consider $n=12t+10$ where $t\geq 1$.  Here, we will extend a block colouring of a $3$-star system of order $n-1\equiv 9$ (mod $12$), as constructed in the proof of Theorem~\ref{cong3_odd}.

Let $V=V' \cup \{x\}$ be the set of points where $|V'|=12t+9$. As in the proof of Theorem~\ref{cong3_odd}, partition $V'$ into $2t$ subsets of size $6$, namely $V_1=\{v_1^1,\ldots,v_6^1\}, \ldots,$ $V_{2t}=\{v_1^{2t},\ldots,v_6^{2t}\}$, and one subset of size $9$, namely $V_{2t+1}=\{v_1^{2t+1},\ldots,v_9^{2t+1}\}$. Let $\mathcal (V',\mathcal B')$ be a $(12t+8)$-block-colourable 3-star system of order $12t+9$ with colour classes labelled as in the proof of Theorem~\ref{cong3_odd}.

The next step is to decompose the edges between the point $x$ and the subsets $V_1,\ldots,V_{2t+1}$ into $3$-stars and colour these. For each $i$ where $1\leq i\leq 2t$, we decompose the edges between the point $x$ and the subset $V_i$ into two $3$-stars $\{x;v_1^i,v_2^i,v_3^i\}$ and $\{x;v_4^i,v_5^i,v_6^i\}$. To colour these, we reuse the colour classes associated to the $1$-factor $F_{i-1}$ from the proof of Theorem~\ref{cong3_odd}, where $F_0=F_{2t}$. 
In the $1$-factor $F_{i-1}$, suppose that $V_i$ is the first element in the ordered pair $(V_i,V_j)\in F_{i-1}$; we can see from the constructions in the proof of Theorem~\ref{cong3_odd} that the $3$-stars $\{x;v_1^i,v_2^i,v_3^i\}$ and $\{x;v_4^i,v_5^i,v_6^i\}$ do not have any intersection with the points in the colour classes $(C_{i-1}^6)'$ and $(C_{i-1}^1)'$ respectively. Therefore, for each $i$ where $1\leq i\leq 2t$, we let $(\mathcal C_{i-1}^1)''= (\mathcal C_{i-1}^1)' \cup \big\{\{x;v_4^{i},v_5^{i},v_6^{i}\}\big\}$ and $(\mathcal C_{i-1}^6)''= \mathcal C_{i-1}^6 \cup \big\{\{x;v_1^{i},v_2^{i},v_3^{i}\}\big\}$. 

It only remains to decompose the edges between the point $x$ and the subset $V_{2t+1}$ into $3$-stars and colour them.  We obtain three $3$-stars $\{x;v_2^{2t+1},v_4^{2t+1},v_7^{2t+1}\}$, $\{x;v_1^{2t+1},v_6^{2t+1},v_8^{2t+1}\}$ and $\{x;v_3^{2t+1},v_5^{2t+1},v_9^{2t+1}\}$. To colour the first two of these, we reuse the colour classes $(C_{2t+1}^1)'$ and $(C_{2t+1}^4)'$, respectively, associated to the $1$-factor $F_{2t+1}$. To colour the third $3$-star, we define a new colour class $\mathcal A_3$ consisting of this $3$-star only.  So we have $(\mathcal C_{2t+1}^1)''=(\mathcal C_{2t+1}^1)' \cup \big\{\{x;v_2^{2t+1},v_4^{2t+1},v_7^{2t+1}\}\big\}$, $(\mathcal C_{2t+1}^4)''=(\mathcal C_{2t+1}^4)' \cup \big\{\{x;v_1^{2t+1},v_6^{2t+1},v_8^{2t+1}\}\big\}$, and $\mathcal A^3=\big\{\{x;v_3^{2t+1},v_5^{2t+1},v_9^{2t+1}\}\big\}$.

Together, these yield a $(12t+9)$-block-colourable $3$-star system of order $12t+10$ with colour classes $(\mathcal C_1^1)'', (\mathcal C_1^2)', \ldots, (\mathcal C_1^5)', (\mathcal C_1^6 )'', \ldots, (\mathcal C_{2t}^1 )'', (\mathcal C_{2t}^2)', \ldots, (\mathcal C_{2t}^5)', (\mathcal C_{2t}^6)'', (\mathcal C_{2t+1}^1 )'', (\mathcal C_{2t+1}^2)',\linebreak (\mathcal C_{2t+1}^3)', (\mathcal C_{2t+1}^4)'', (\mathcal C_{2t+1}^5)', (\mathcal C_{2t+1}^6)', \mathcal A^1, \mathcal A^2, \mathcal A^3.$
\end{proof}

Combining the results of this section, as well as the $e=3$ case of Corollary~\ref{combined_e}, we have the following corollary.

\begin{cor} \label{combined_3}
For every admissible order $n$, there exists either an $(n-1)$-block-colourable or an $n$-block-colourable $3$-star system of order $n$. 
\end{cor}

\section{Conclusion} \label{sec:conclusion}
Admittedly, as can be seen from the experimental results in Section~\ref{sec:computer} (particularly Table~\ref{table:chromatic}), our main theorems in Sections~\ref{sec:bound-e} and~\ref{sec:bound-3} do not give the best possible bounds on the least possible chromatic index of an $S_e(n)$.  However, a more positive observation is that our upper bounds are linear in $n$, as is the trivial lower bound of $L(n,e)$ at the end of Section 1.  So we can conclude that the least possible chromatic index is asymptotically $\Theta(n)$, while the number of blocks is quadratic in $n$.  However, the general question of precisely determining the least possible chromatic index of an $S_e(n)$ is, of course, still wide open.

Our computations also showed that, for $e=3$ and $9\leq n\leq 10$, the {\em largest} possible chromatic index of an $S_e(n)$ is equal to the number of blocks, i.e.\ any two blocks must intersect (such as the system in Example~\ref{rainbow}).  Determining what this maximum value can be is another interesting open problem.  (We note that for $K_3$-designs, i.e.\ Steiner triple systems, the equivalent question was discussed by Rosa~\cite{Rosa15} in 2015 and considered in detail by Bryant {\em et al.}\ in 2017~\cite{Bryant17}.)

\section*{Appendix A: GAP programs}

\subsection*{A.1\quad General techniques}
In {\sf GAP}, we specify an $e$-star as an ordered pair of sets, where the first entry is the set containing the root vertex, and the second entry is the set of pendant vertices.  So, for example, the $3$-star $\{1; 2,3,4\}$ is given as {\footnotesize{\texttt{[ [1], [2,3,4] ]}}}.  An $e$-star system is specified as a set of $e$-stars. The following functions are used to manipulate $e$-stars or $e$-star systems, and to determine the chromatic index of a system (as the chromatic number of its block intersection graph).  Most of these functions are dependent on the {\sf GRAPE} package, which must be accessed using the {\footnotesize{\texttt{LoadPackage("grape");}}} command.

\begin{scriptsize}
\begin{verbatim}
## Function to make the edge set of an e-star <s>
StarEdges:=function(s,e)
  local E;
  E:=List(s[2], x->Set([ s[1][1], x ]) );
  return Set(E);
end;

## Inverse function to make an e-star from its edge set
InverseStarEdges:=function(edge_set,e)
  local root,pendants,S;
  root:=Intersection(edge_set);
  if Size(root)<>1 
    then return fail;
    else
    pendants:=List(edge_set, x->Difference(x,root));
    S:=[root, Union(pendants)];
    return S;
  fi;
end;

## Function to determine the chromatic index of an e-star system <D>
ChromaticIndexOfStarSystem:=function(D)
  local big,auts,chi;
  D:=Set(D);
  big:=Graph(	Group(()), D, OnTuplesSets,  
              function(x,y) return (x<>y and Intersection(Union(x),Union(y))<>[]); end);
  auts:=AutGroupGraph(big);
  big:=NewGroupGraph(auts,big); ## ensures the full automorphism group of <big> is used  
  chi:=ChromaticNumber(big);
  return chi;
end;
\end{verbatim}
\end{scriptsize}

\subsection*{A.2\quad The clique-finding approach}
Here, we give the {\sf GAP} code used in subsection~\ref{sec:cliques}, using a clique-finding approach to obtain the chromatic index of all $S_3(9)$.

\begin{scriptsize}
\begin{verbatim}
## Function to make the graph whose vertices are all possible e-stars, adjacent whenever they are edge-disjoint
StarGraph:=function(n,e)
  local gamma;
  gamma:=Graph( SymmetricGroup(n), [ [[1],[2..e+1]] ], OnTuplesSets, 
                function(x,y) return ( x<>y and Intersection(StarEdges(x,e),StarEdges(y,e))=[] ); end );
  return gamma;
end;

## Function to convert a clique <C> in a "star graph" <gamma> into a set of e-stars
CliqueStars:=function(gamma,C)
  local blocks;
  blocks:= List(C, x->gamma.names[x]); ## the ".names" component stores the e-star corresponding to each vertex
  return blocks;
end;

## Code to obtain all possible 3-star systems of order 9, and determine the chromatic index of each
gamma:=StarGraph(9,3);
cliques:=CliquesOfGivenSize(gamma,12,1); ## the argument "1" will find all cliques of the specified size
chromatic_list:=[];
for C in cliques do
  D:=CliqueStars(gamma,C);
  chi:=ChromaticIndexOfStarSystem(D);
  Add(chromatic_list,chi);
od;
\end{verbatim}
\end{scriptsize}

\subsection*{A.3\quad The design classification approach}
Here, we give the {\sf GAP} code (using the {\sf DESIGN} package~\cite{DESIGN}) implementing the techniques described in subsection~\ref{sec:soicher}.

\begin{scriptsize}
\begin{verbatim}
## Obtain action of Sym(n) on edges of complete graph K_n
SymmetricGroupOnEdges:=function(n)
  local S,edges,hom;
  S:=SymmetricGroup([1..n]);
  edges:=Combinations([1..n],2);
  hom:=ActionHomomorphism(S,edges,OnSets);
  return Image(hom,S);
end;

## Function to obtain representatives for each conjugacy class of subgroups of Sym(n) of prime order
PGroupRepresentatives:=function(n)
  local G,class_reps,group_list,c,x;
  G:=SymmetricGroupOnEdges(n);
  class_reps:=Set(ConjugacyClasses(G),c->Group(Representative(c)));
  group_list:=Filtered(class_reps,x->IsPrimeInt(Size(x)));
  return group_list;
end;

## Function to construct set of all e-stars as a block design
## Note that entries 1..e of "edges" will be the edges of the e-star [ [1], [2..e+1] ]
AllStarsDesign:=function(n,e)
  local edges,G,D;
  edges:=Combinations([1..n],2);
  G:=SymmetricGroupOnEdges(n);
  D:=BlockDesign(Size(edges),[[1..e]],G);;
  AllTDesignLambdas(D);
  return D;
end;



## Function to obtain all e-star systems of order n invariant under a given group H
InvariantStarSystems:=function(n,e,H)
  local edges,G,D,N,spreads,systems_edges,systems,x,y,z;
  edges:=Combinations([1..n],2);
  G:=SymmetricGroupOnEdges(n);
  D:=AllStarsDesign(n,e);
  N:=Normalizer(G,H);
  spreads:=BlockDesigns(rec(v:=D.v, blockSizes:=BlockSizes(D),
                            blockDesign:=D, tSubsetStructure:=rec(t:=1, lambdas:=[1]),
                            requiredAutSubgroup:=H, isoGroup:=N, isoLevel:=2));
  systems_edges:=List(spreads, x->List(x.blocks, y->edges{y}));
  systems:=List(systems_edges, x->Set(List(x, y->InverseStarEdges(y,3))));
  return systems;
end;

## Sample code to determine chromatic index of 3-star systems of order 10 invariant under groups of prime order
systems:=[];
groups:=PGroupRepresentatives(10);
for H in groups do
  Append(systems, InvariantStarSystems(10,3,H));
od;
chromatic_list:=[];
for D in systems do
  chi:=ChromaticIndexOfStarSystem(D);
  Add(chromatic_list,chi);
od;

\end{verbatim}
\end{scriptsize}

\subsection*{Acknowledgements}
The authors would like to thank Leonard Soicher for suggesting the computational approach of subsection~\ref{sec:soicher}.  They also acknowledge financial support from an NSERC Discovery Grant held by the first author, in particular the one-time, one-year extension with funds due to COVID-19.  Finally, we thank the two anonymous referees whose comments helped to improve the paper.

\end{document}